\documentclass[11pt, a4paper, reqno]{amsart}
\usepackage{amssymb, array, amsmath, amscd, pdfpages, enumerate, amsthm, fixltx2e, setspace, hyperref,mathtools}

\usepackage[utf8]{inputenc}
\usepackage{amssymb}

\usepackage{graphicx, caption} 
\usepackage{tikz}
\usepackage{bm}

\DeclareMathOperator{\spec}{sp}
\DeclareMathOperator{\spsup}{sp \;sup}
\DeclareMathOperator{\spinf}{sp \;inf}
\DeclareMathOperator{\strsup}{str \;sup}
\DeclareMathOperator{\strinf}{str \;inf}

\newcommand*{\R}{{\mathbb{R}}}     
\newcommand*{\Z}{{\mathbb{Z}}}     
     
\newcommand*{\N}{{\mathbb{N}}}

\newcommand*{\Abs}[2][default]{\ifthenelse{\equal{#1}{default}}{\left\lvert#2\right\rvert}{\ldelim{#1}{\lvert}#2\rdelim{#1}{\rvert}}}
\newcommand*{\Norm}[2][default]{\ifthenelse{\equal{#1}{default}}{\left\lVert#2\right\rVert}{\ldelim{#1}{\lVert}#2\rdelim{#1}{\rVert}}}

\newcommand*{\Iprod}[3][default]{\ifthenelse{\equal{#1}{default}}{\left\langle#2,#3\right\rangle}{\ldelim{#1}{\langle}#2,#3\rdelim{#1}{\rangle}}}
\newcommand*{\Dualpair}[3][default]{\ifthenelse{\equal{#1}{default}}{\left\langle#2,#3\right\rangle}{\ldelim{#1}{\langle}#2,#3\rdelim{#1}{\rangle}}}

\newcommand*{\ddb}[2][1]{\ifthenelse{\equal{#1}{1}}{\frac{d}{d#2}}{\frac{d^{#1}}{d#2^{#1}}}}
\newcommand*{\pd}[3][1]{\ifthenelse{\equal{#1}{1}}{\frac{\partial{#2}}{\partial{#3}}}{\frac{\partial^{#1}{#2}}{\partial#3^{#1}}}}

\newcommand{\blue}[1]{{\color{blue} #1}}
\renewcommand{\blue}[1]{#1}

\newcommand*{\rdc}[1]{{#1}^{\perp}}
\newcommand*{\ldc}[1]{{}^{\perp}#1}
\newcommand*{\aspp}{({A}^{\perp})_{sp}}

\newcommand*{\apld}{({}^{\perp}A)_{sp}}

\newcommand*\lenv{{\hbox{\raisebox{-.15ex}{\rotatebox[origin=c]{50}{$\smallsmile$}}\kern-8.65pt\rotatebox[origin=c]{-25}{$\smallsetminus$}}}}
\newcommand*\uenv{{\hbox{\raisebox{-.0ex}{\rotatebox[origin=c]{-45}{$\smallfrown$}}\kern-5.6pt\raisebox{.2ex}{\rotatebox[origin=c]{-5}{\scriptsize\slash}}}}\,\kern+1.5pt}
\newcommand{\sleq}{\preccurlyeq}
\newcommand{\sgeq}{\succcurlyeq}
 
\newcommand*{\Ss}{{\mathcal{S}}}

\newcommand*{\lupp}[1]{\kern0pt^{\kern0pt u \kern0pt}\kern1pt#1}
\newcommand*{\llow}[1]{\kern0pt^{\kern0pt l \kern0pt}\kern1pt#1}
\newcommand*{\rupp}[1]{#1\kern1pt^{\kern0pt u \kern0pt}\kern0pt}
\newcommand*{\rlow}[1]{#1\kern1pt^{\kern0pt l \kern0pt}\kern0pt}
\newcommand*{\ul}[1]{\kern0pt^{\kern0pt u \kern0pt}\kern0pt#1\kern1pt^{\kern0pt l \kern0pt}\kern0pt}
\newcommand*{\lu}[1]{\kern0pt^{\kern0pt l \kern0pt}\kern0pt#1\kern0pt^{\kern0pt u \kern0pt}\kern0pt}
\newcommand*{\bproofname}{Proof}
\newenvironment{bproof}[1][\bproofname]{\begin{proof}[#1]}{\end{proof}}

\newtheorem{thm}{Theorem}[section]
\numberwithin{thm}{section}
\newtheorem{prop}[thm]{Proposition}
\newtheorem{lemma2}[thm]{Lemma}
\newtheorem{cor}[thm]{Corollary}
\theoremstyle{definition}
\newtheorem{defn}[thm]{Definition}
\newtheorem{remark2}[thm]{Remark}
\newtheorem{example2}[thm]{Example}
\numberwithin{equation}{section}

\begin{document}

\title[Ideals, bands and decompositions in mixed lattice spaces]{Ideals, bands and direct sum decompositions in mixed lattice vector spaces}

\thispagestyle{plain}

\author{Jani Jokela}
\address[J. Jokela]{Mathematics, Faculty of Information Technology and Communication Sciences, 
Tampere University, PO. Box 692, 33101 Tampere, Finland}
\email{jani.jokela@tuni.fi}


\begin{abstract}
A mixed lattice vector space is a partially ordered vector space with two partial orderings and certain lattice-type properties. 
In this paper we first give some fundamental results in mixed lattice groups, and then we investigate the structure theory of mixed lattice vector spaces, which can be viewed as a generalization of the theory of Riesz spaces. More specifically, we study the properties of ideals and bands in mixed lattice spaces, and the related idea of representing a mixed lattice space as a direct sum of disjoint bands. Under certain conditions, these decompositions can also be given in terms of order projections. 
\end{abstract}

\maketitle

\section{Introduction}
\label{sec:s1}
The idea of equipping a vector space with two partial orderings plays an important role in classical potential theory \cite{doob}. 
During the 1970s and early 1980s, M. Arsove and H. Leutwiler 
introduced the notion of a mixed lattice semigroup which provides a rather general setting for axiomatic potential theory \cite{ars}. 
The novelty of their theory was that it mixed two partial orderings in a semigroup in such way that the resulting structure is not a lattice, in general, but it has many lattice-type properties, and the interplay between the two partial orderings plays a fundamental role in the theory. Although Arsove and Leutwiler formulated their theory in the semigroup setting, a similar mixed lattice order structure can also be imposed on groups and vector spaces where non-positive elements are present. This gives rise to the notions of mixed lattice groups and mixed lattice vector spaces. In this paper, we are mostly concerned with the latter.   
A mixed lattice space is a real partially ordered vector space with two partial orderings, where the usual lattice operations (i.e. the supremum and infimum of two elements) are replaced with asymmetric mixed envelopes that are formed with respect to the two partial orderings. As a consequence, the theory of mixed lattice structures is asymmetric in nature. For example, the mixed envelopes do not have commutative or distributive properties, in contrast to the theory of Riesz spaces, where the lattice operations are commutative and distributive. 
On the other hand, mixed lattice spaces have also many similarities to Riesz spaces. In fact, a Riesz space is just a special case of a mixed lattice vector space, in which the two partial orderings coincide. Early studies of mixed lattice groups were done by Eriksson-Bique \cite{eri1}, \cite{eri}, and more recently by the authors in \cite{jj1}. The present paper is a continuation of the research that was commenced in \cite{eri}, \cite{eri1} and \cite{jj1}.

Ideals and bands are the main structural components of a Riesz space, and Riesz spaces can be decomposed into order direct sums of disjoint bands.  
Apart from Riesz spaces, these concepts have recently been studied in more general ordered vector spaces \cite{prerieszbands,povsbands}, as well as in mixed lattice semigroups in \cite{ars}. 
It is therefore natural to explore these ideas also in the mixed lattice space setting, and this is the main topic of the present paper. 

First we give a brief survey of terminology, definitions and basic results that will be needed in the subsequent sections.
Section \ref{sec:s3} contains some important results on the properties of the mixed envelopes. Many of these results are known in mixed lattice semigroups through the work of Arsove and Leutwiler in \cite{ars}, but they haven't been studied in the group setting. The main difficulty  here is that many of the results that hold in a mixed lattice semigroup depend on the fact that all the elements are positive. 
Therefore, some of these properties do not hold in the group setting without some limitations.   
We also give the mixed lattice version of the fundamental dominated decomposition theorem.

The concept of an ideal in mixed lattice groups and vector spaces was introduced in \cite{jj1}. 
Many properties of the most relevant subspaces, including ideals, are determined by their sets of positive elements.  
In this context, positive cones play a fundamental role and we discuss some properties of mixed lattice cones, which generate mixed lattice subspaces. 
In Section \ref{sec:s4} we study the structure of ideals, and we introduce the notion of a band in mixed lattice spaces and give some basic results concerning bands.  
The two partial orderings give rise to two different types of ideals, called ideals and specific ideals, 
depending on which partial ordering is considered. 
Consequently, we also have different notions of a band,  
depending on the type of the underlying ideal. In addition to these, we introduce an intermediate notion of a quasi-ideal, which plays an important role in the structure theory.  

Disjoint complements are introduced in Section \ref{sec:s5}. 
The disjointness of two positive elements is defined in a similar manner as in Riesz spaces, by requiring that the lower envelope of the two elements is zero. However, the asymmetric nature of the mixed envelopes leads to two distinct one-sided notions of disjointness, the left and right disjointness.
The disjoint complements and direct sum decompositions in mixed lattice semigroups were studied by Arsove and Leutwiler in \cite{ars}, where they showed that the left and right disjoint complements of any non-empty set in a mixed lattice semigroup have certain band-type properties. In mixed lattice vector spaces the situation is a bit more complicated, again due to the existence of non-positive elements. In order to obtain a satisfactory theory, we consider disjointness first 
for positive elements, 
and we then define the left and right disjoint complements as the specific ideals and ideals, respectively, generated by certain cones of positive elements that satisfy the disjointness conditions. This approach turns out to be quite natural, and we also find it to be compatible with the existing theory of mixed lattice semigroups. 
We then investigate the structure and properties of the disjoint complements, such as the band properties and the  
structure of the sub-cones that generate the disjoint complements.   
We also introduce the notion of symmetric absolute value in mixed lattice spaces and show that the right disjoint complement can be described in terms of the absolute value. This approach corresponds to the usual definition of the disjoint complement in a Riesz space. Finally, we give some results concerning bands in Archimedean mixed lattice spaces.

In the final section we study decompositions of mixed lattice spaces, where the space can be written as a direct sum of an ideal and a specific ideal that are disjoint complements of each other. 
The direct sum decompositions give rise to band projections which can be expressed in terms of order projection operators. Of course, these projection properties have their well-known analogues in Riesz space theory. 
For an account of the theory and terminology of Riesz spaces, we refer to \cite{lux} and \cite{zaa}.

\section{Preliminaries}
\label{sec:s2}


The fundamental structure on which most of the theory of mixed lattice groups and vector spaces is based on, is called a \emph{mixed lattice semigroup}. Let $S(+,\leq)$ be a positive partially ordered abelian semigroup with zero element. The semigroup $S(+,\leq)$ is also assumed to have the cancellation property: 
$x+z\leq y+z$ implies $x\leq y$ for all $x,y,z\in S$. 
The partial order $\leq$ is called the \emph{initial order}. A second partial ordering $\sleq$ (called the \emph{specific order}) is then defined on $S$ by $x\sleq y$ if $y=x+m$ for some $m\in S$. Note that this implies that $x\sgeq 0$ for all $x\in S$. 
The semigroup $(S,+,\leq,\sleq)$ is called a \emph{mixed lattice semigroup} if, in addition, the \emph{mixed lower envelope}
$$
x\lenv y \,=\,\max \,\{\,w\in S: \; w\sleq x \; \textrm{ and } \; w\leq y \,\} %
$$
and the \emph{mixed upper envelope}
$$
x\uenv y\,=\,\min \,\{\,w\in S: \; w\sgeq x \; \textrm{ and } \; w\geq y \,\},
$$ 
exist for all $x,y\in S$, and they satisfy the identity 
\begin{equation}\label{basic_identity}
x\uenv y + y\lenv x =x+y.
\end{equation} 
In the above expressions the minimum and maximum are taken with respect to the initial order $\leq$.

Typical examples of mixed lattice semigroups are constructed by starting with some vector space $V$ and a cone $C$ in $V$ which generates a partial ordering $\leq$. Then a sub-cone $S$ of $C$ is taken as the semigroup in which the specific order $\sleq$ is defined as the partial ordering induced by the semigroup $S$ itself, that is, $S$ is taken as the positive cone for the partial ordering $\sleq$. This procedure  turns $S$ into a mixed lattice semigroup, provided that the conditions in the above definitions are satisfied. Many examples of mixed lattice structures are given in \cite{jj1}.

A \emph{mixed lattice group} is a partially ordered commutative group $(G,+,\leq,\sleq)$ with two partial orderings $\leq$ and $\sleq$ (called again the initial order and the specific order, respectively) such that the mixed upper and lower envelopes $x\uenv y$ and $x\lenv y$, as defined above, exist in $G$ for all $x,y\in G$. A sub-semigroup $S$ of a mixed lattice group $G$ is called a \emph{mixed lattice sub-semigroup} of $G$ if $x\uenv y$ and $x\lenv y$ belong to $S$ whenever $x,y\in S$. 
It should be remarked here that the additional identity \eqref{basic_identity} in the definition of a mixed lattice semigroup is then automatically satisfied because the identity holds in $G$ (see further properties of the mixed envelopes below). 
A mixed lattice group $G$ is called \emph{quasi-regular} if the set $G_{sp}=\{x\in G:x\sgeq 0\}$ is a mixed lattice sub-semigroup of $G$, and $G$ is called \emph{regular} if $G$ is quasi-regular and every $x\in G$ can be written as $x=u-v$ where $u,v\in G_{sp}$.

There are several fundamental rules for the mixed envelopes that hold in quasi-regular mixed lattice groups. These are listed below, and will be used in this paper frequently. For more details concerning these rules and their proofs we refer to \cite{jj1} and \cite{eri}.

The following hold for all elements $x$, $y$ and $z$ in a quasi-regular mixed lattice group $G$.

$$ 
\begin{array}{ll}\label{p1to8}
(M1)   & \; x\uenv y \, + \, y \lenv x \, = \, x+y  \\ 
(M2a)  & \; z \,+ \,x\uenv y \, = \, (x+z)\uenv(y+z) \\ 
(M2b)  & \; z \,+\, x\lenv y \, = \, (x+z)\lenv(y+z) \\
(M3)   & \; x\uenv y \, = \, -(-x \lenv -y) \\
(M4)  & \; x\lenv y \sleq x\sleq x\uenv y \quad \textrm{and} \quad x\lenv y \leq y\leq x\uenv y \\
(M5)  & \; x\sleq u \quad \textrm{and} \quad y\leq v \; \implies \; x\uenv y \leq u\uenv v \quad \textrm{and} \quad x\lenv y \leq  u\lenv v \\
(M6a)  & \; x\sleq y \, \iff \, x\uenv y = y \, \iff \, y\lenv x = x \\
(P6b)  & \; x\leq y \, \iff \, y\uenv x = y \, \iff \, x\lenv y = x \\
(M7)   & \; x\sleq y \, \implies \, x\leq y  \\
(M8a)  & \; x\sleq z \quad \textrm{and} \quad y\sleq z \; \implies \; x\uenv y \sleq z \\
(M8b)  & \; z\sleq x \quad \textrm{and} \quad z\sleq y \; \implies \; z \sleq x\lenv y  \bigskip\\
\end{array} 
$$

We should mention here that (M1)--(M6a) hold more generally in every mixed lattice group. Moreover, quasi-regularity is equivalent to the properties (M8a) and (M8b), and these in turn imply (M7) (but not conversely). A mixed lattice group in which (M7) holds is called \emph{pre-regular}.

\section{Additional properties of mixed lattice groups}
\label{sec:s3}

In this section we give several further properties of mixed lattice groups. Many of these have been studied in mixed lattice semigroups in \cite{ars}, but some of them do not carry over to the group setting without some restrictions.

We begin by introducing a useful tool (which is due to M. Arsove and H. Leutwiler, \cite{ars}) for studying the properties of the mixed envelopes. For $x,u\in G$ we define the mapping $S_x(u)=\min\{w\sgeq 0:u\leq w+x\}$, where the minimum is taken with respect to $\leq$. We will first show that the element $S_x(u)$ exists for all $x,u\in G$ and it has certain basic properties.

\begin{prop}\label{s_u}
Let $G$ be a quasi-regular mixed lattice group. The mapping $S_x$ defined above has the following properties. 
\begin{enumerate}[(a)]
\item
The element $S_x(u)=\min\{w\sgeq 0:u\leq w+x\}$ exists for all $x,u\in G$ and it is given by $S_x(u)=x\uenv u -x=u-u\lenv x$.
\item
If $x\sgeq 0$  then\,  $S_x S_y (u)=  S_{x+y}(u)$ \, holds for all $u,y\in G$.  Moreover, if $x,y\sgeq 0$  then $S_x S_y (u)=S_y S_x (u)=S_{x+y} (u)$ for all $u\in G$.
\item
$S_{x+a}(u+a)=S_x(u)$ \,  for all \, $x,u,a\in G$.
\item
If\, $u\sleq v$ \,then\, $S_x(u)=S_{v-u}S_x(v)$ \, for all  $x\in G$.
\end{enumerate}
\end{prop}

\begin{bproof}
\begin{enumerate}[(a)]
\item
Let $A=\{w\sgeq 0:u\leq w+x\}$. Let $m=x\uenv u -x$.
The equality $x\uenv u -x=u-u\lenv x$ follows immediately from (M1). First we note that 
$m\sgeq 0$ and
$$
m+x=x\uenv u -x+x=x\uenv u \geq u,
$$
and so $m\in A$. Assume then that $z\in A$. Then $u\leq z+x$ and since $z\sgeq 0$, we also have $u\sleq z+u$. By (M5) and (M2) it now follows that $u\leq (z+u)\lenv (z+x) = z+ u\lenv x$ and this implies that $m=u-u\lenv x\leq z$. This shows that $m=S_x(u)=\min\{w\sgeq 0:u\leq w+x\}$.
\item 
By part (a) we have $u\leq y\uenv u= (y\uenv u -y)+y =S_y(u)+y$ for all $y,u\in G$, and similarly $S_y(u)\leq S_x S_y(u) +x$. It follows that $u\leq S_x S_y(u) +x+y$ and by the definition of $S_x$ we obtain the inequality $S_{x+y} (u)  \leq S_x S_y (u)$. Similarly, exchanging the roles of $x$ and $y$ gives $S_{x+y} (u)  \leq S_y S_x (u)$.
On the other hand, we have $u\leq S_{x+y} (u) +x+y$, and if $x\sgeq 0$ then $S_{x+y} (u) +x\sgeq 0$, and it follows again from the definition that $S_y(u) \leq S_{x+y} (u) +x$. Hence, $S_x S_y (u)\leq S_{x+y} (u)$ and so the equality $S_x S_y (u)=  S_{x+y}(u)$  holds for all $u,y\in G$. If also $y\sgeq 0$ then we can exchange the roles of $x$ and $y$ above to get $S_x S_y (u)=S_y S_x (u)= S_{x+y} (u)$ for all $u\in G$.
\item 
The translation invariance property follows immediately from the definition of $S_x$. 
\item 
Let $u\sleq v$. Then $a=v-u\sgeq 0$ and using (b) and (c) gives
$$
S_x(u)=S_{x+a}(u+a)=S_aS_x(u+a)=S_{v-u}S_x(v).
$$
\end{enumerate}
\end{bproof}

Using the above results for $S_x$ we can prove several additional properties of the mixed envelopes. We begin with the following theorem, which gives two important inequalities for the mixed envelopes. They were first studied in mixed lattice semigroups by Arsove and Leutwiler \cite[Theorem 3.2]{ars}. Later, Eriksson--Bique showed that they hold also in a group extension of a mixed lattice semigroup, and therefore in every regular mixed lattice group \cite[Theorem 3.5]{eri1}. We will now prove that the inequalities hold more generally in every quasi-regular mixed lattice group. Later on, these inequalities will play an important role in the development of the theory.

\begin{thm}\label{new_ineq}
Let $G$ be a quasi-regular mixed lattice group. If $x\sleq y$ then the inequalities
$$
u\lenv x \sleq u\lenv y \qquad \text{and} \qquad u\uenv x \sleq u\uenv y \qquad 
$$
hold for all $u\in G$.
\end{thm}

\begin{bproof}
First we observe that the first inequality is equivalent to
$$
0 \sleq u\lenv y -u\lenv x = (u-u\lenv x)-(u-u\lenv y) = S_x(u)-S_y(u). 
$$
Now if $x\sleq y$ then 
$y=a+x$, where $a=y-x\sgeq 0$. 
Using the properties of the mapping $S_x$ we get
$$
S_y(u)=S_{x+a}(u)=S_a(S_x(u))=S_{y-x}(S_x(u))=S_x(u)-S_x(u)\lenv (y-x),
$$
and this yields 
$$
S_x(u)-S_y(u)= 
S_x(u)\lenv (y-x)=(u-u\lenv x)\lenv (y-x).
$$
Hence we have  $u\lenv y -u\lenv x =(u-u\lenv x)\lenv (y-x)$. Now $u-u\lenv x\sgeq  0$ and $y-x\sgeq 0$, and since $G$ is quasi-regular, 
it follows by (M8b) that $(u-u\lenv x)\lenv (y-x)\sgeq 0$. Thus we have shown that  $u\lenv y -u\lenv x \sgeq 0$, or $u\lenv x \sleq u\lenv y$.

The proof of the other inequality is similar.
\end{bproof}

The next example shows that the assumptions in the last theorem cannot be weakened, that is, the inequalities do not hold if $G$ is not quasi-regular.

\begin{example2}\label{prereg_ei_p7}
Let $G=(\Z\times \Z,\leq,\sleq)$ and define partial orders $\leq$ and $\sleq$ as follows. If $x=(x_1,x_2)$ and $y=(y_1,y_2)$ then $x\leq y$ iff $x_1\leq y_1$ and $x_2\leq y_2$. In addition, $x\sleq y$ iff $x=y$ or $y_1\geq x_1+1$ and $y_2\geq x_2+1$. 
Then $G$ is a mixed lattice group which is pre-regular but 
not quasi-regular (see \cite[Example 2.23]{jj1}). 
If $x=(1,0)$, $y=(2,1)$ and $u=(0,0)$ then $u\uenv x=(1,1)$ and $u\uenv y=(2,1)$. Now $x\sleq y$ but $u\uenv x \sleq u\uenv y$ does not hold.
\end{example2}

The inequalities in Theorem \ref{new_ineq} have several implications. For example, the following one-sided associative and distributive laws hold for the mixed envelopes. The proof is exactly the same as in the mixed lattice semigroup case, see \cite[Theorems 3.3 and 3.5]{ars}.

\begin{thm}
If $G$ is a quasi-regular mixed lattice group then the following one-sided associative laws
$$
(x\lenv y)\lenv z \geq x\lenv (y\lenv z) \quad \textrm{and} \quad (x\uenv y)\uenv z \leq x\uenv (y\uenv z)
$$
and one-sided distributive laws
$$
x\lenv(y\uenv z) \geq (x\lenv y)\uenv (x\lenv z) \quad \textrm{and} \quad 
x\uenv(y\lenv z) \leq (x\uenv y)\lenv (x\uenv z)
$$
hold for all $x,y,z\in G$.
\end{thm}


\blue{ 
In the sequel, $\sup E$ and $\inf E$ stand for the supremum and infimum of a subset $E$ with respect to the initial order $\leq$. The supremum and infimum of $E$ with respect to the specific order $\sleq$ are denoted by $\spsup E$ and $\spinf E$, respectively. 

The following result is originally due to N. Boboc and A. Cornea (see \cite[Proposition 2.1.4]{boboc}). The more general mixed lattice version is given in \cite[Theorem 4.2]{ars}.

\begin{lemma2}\label{boboc}
Let $E$ be a subset of a quasi-regular mixed lattice group $G$ such that $u_0=\sup E$ exists in $G$. If $w\in G$ is an element such that $x\sleq w$ for all $x\in E$ then $u_0\sleq w$. Similarly, if $v_0=\inf E$ exists and $w\in G$ is such that $w\sleq x$ for all $x\in E$ then $w\sleq v_0$.
\end{lemma2}

\begin{bproof}
If $u_0=\sup E$ then $x\sleq u_0$ for all $x\in E$, and so if $x\sleq w$ for all $x\in E$ then by (M5), $x\leq w\lenv u_0 \leq u_0$ for all $x\in E$. This implies that $u_0=w\lenv u_0\sleq w$. The result concerning the infimum can be proved by a similar argument.
\end{bproof}

We obtain stronger results if the supremum and specific supremum of a subset are equal. 

\begin{defn}
Let $E$ be a subset of a mixed lattice group $G$ such that $\sup E$ and $\spsup E$ exist in $G$, and $u_0=\sup E = \spsup E$. The element $u_0$ is called the \emph{strong supremum} of $E$ and it is denoted by $u_0 =\strsup E$. The \emph{strong infimum} $v_0$ of $E$ is defined similarly, and it is denoted by $v_0=\strinf E$. 
\end{defn}

The following result can now be proved exactly the same way as in the theory of mixed lattice semigroups (see \cite{ars}, pp.23). We include the proof here for completeness.

\begin{prop}\label{strongsup_properties}
Let $E$ be a subset of a quasi-regular mixed lattice group $G$ such that $u_0=\strsup E$ exists in $G$. 
Then for all $x\in G$ 
$$
\strsup_{u\in E}(x\lenv u)=x\lenv u_0 \quad \text{and} \quad \strsup_{u\in E}(x\uenv u)=x\uenv u_0.
$$ 
Similarly, if $v_0=\strinf E$ exists then for all $x\in G$ 
$$
\strinf_{u\in E}(x\lenv u)=x\lenv v_0 \quad \text{and} \quad \strinf_{u\in E}(x\uenv u)=x\uenv v_0.
$$ 
\end{prop}

\begin{bproof}
If $x\in G$ and $u_0=\strsup E$ then By (M5), $x\lenv u\leq x\lenv u_0$ and $u\uenv x \leq u_0 \uenv x$ for all $u\in E$. Let $w$ be any element such that $x\lenv u\leq w$ for all $u\in E$. Then by (M1) we have 
$$
x+u=x\lenv u + u\uenv x  \leq x\lenv u + u_0\uenv x \leq w + u_0\uenv x  
$$
for all $u\in E$. Thus $w + u_0\uenv x -x$ is a $(\leq)$-upper bound of $E$, and since $u_0=\sup E$, by (M1) this implies that 
$$
x\lenv u_0 + u_0\uenv x = u_0+x  \leq (w + u_0\uenv x -x)+x=w + u_0\uenv x.
$$ 
Hence $x\lenv u_0 \leq w$, and this shows that $x\lenv u_0 = \sup  \{x\lenv u:u\in E\}$. On the other hand, by Theorem \ref{new_ineq} we have $x\lenv u\sleq x\lenv u_0$ for all $u\in E$. If $x\lenv u\sleq w$ for all $u\in E$ then it follows by Lemma \ref{boboc} that $x\lenv u_0\sleq w$, and so $x\lenv u_0 = \spsup  \{x\lenv u:u\in E\}$. This shows that $x\lenv u_0 = \strsup  \{x\lenv u:u\in E\}$. The other identities can be proved in a similar manner.
\end{bproof}
}


\blue{
In the sequel, we will use the following notation. The inequalities $x\sleq y$ and $y\leq z$ are written more concisely as $x\sleq y \leq z$. Similarly, the notation $x\leq y \sleq z$ means that $x\leq y$ and $y\sleq z$.
}

We conclude this section by showing that mixed lattice groups have the dominated decomposition property. 
We actually have different variants of this property that will be useful.

\begin{thm}\label{riesz_dec}
Let $G$ be a quasi-regular mixed lattice group.
\begin{enumerate}[(a)]
\item
Let $u\sgeq 0$,\, $v_1\geq 0$\, and \,$v_2\sgeq 0$\, be elements of $G$ satisfying $u\leq v_1+v_2$. Then there exist elements $u_1$ and $u_2$ 
such that \,$0\leq u_1\leq v_1$, \;$0\sleq u_2\leq v_2$\; and \; $u=u_1+u_2$. Moreover, if $v_1\sgeq 0$ then $u_1 \sgeq 0$, and if $u\sleq v_1+v_2$ then $u_2\sleq v_2$.
\item
Let $u\geq 0$,\, $v_1\sgeq 0$\, and \,$v_2\geq 0$\, be elements of $G$ satisfying $u\sleq v_1+v_2$. Then there exist elements $u_1$ and $u_2$  
such that \,$0\leq u_1\sleq v_1$, \;$0\leq u_2\leq v_2$\; and \; $u=u_1+u_2$. Moreover, if $u\sgeq 0$ then $u_1 \sgeq 0$. 
\end{enumerate}
\end{thm}

\begin{bproof}
\emph{(a)}\;
The element $u_1=u\lenv v_1$ satisfies $u_1\leq v_1$ and $u_1\geq 0$ (and $u_1\sgeq 0$ if $v_1\sgeq 0$). 
Let $u_2=u-u_1$. Then $u=u_1+u_2$ and $u_2\sgeq 0$, since $u_1\sleq u$. It remains to show that $u_2\leq v_2$. For this, we note that $0\sleq v_2$ and $u-v_1\leq v_2$, and so we have
$$
u_2=u-u\lenv v_1 = u+ (-u)\uenv(-v_1) = 0\uenv (u-v_1) \leq v_2\uenv v_2 =v_2.
$$
If $u\sleq v_1+v_2$ then $u-v_1\sleq v_2$ and we can replace $\leq$ by $\sleq$ in the above inequality, by (M8a).

\emph{(b)}\;
The element $u_1=v_1\lenv u$ satisfies $u_1\leq u$, $u_1\sleq v_1$ and $u_1\geq 0$ (and $u_1\sgeq 0$ if $u\sgeq 0$). It follows from  
$u\sleq v_1+v_2$ and $v_1\leq v_1+v_2$ that $u\uenv v_1\leq v_1+v_2$, and so if we set $u_2=u-u_1$ we have $u_2\geq 0$ and
$$
u_2=u-v_1\lenv u = u\uenv v_1 - v_1 \leq v_1+v_2 -v_1=v_2,
$$ 
and the proof is complete.
\end{bproof}

\section{Ideals and bands in a mixed lattice vector space}
\label{sec:s4}

A \emph{mixed lattice vector space} \blue{(or briefly, a \emph{mixed lattice space})} was defined in \cite{jj1} as a partially ordered real vector space $V$ with two partial orderings $\leq$ and $\sleq$ (called the initial order and the specific order, respectively) such that the mixed upper and lower envelopes $x\uenv y$ and $x\lenv y$, as defined in Section \ref{sec:s2}, exist in $V$ for all $x,y\in V$.

We recall that a subset $C$ of a vector space is called a \emph{cone} if \, (i)\;$\alpha C \subseteq C$ for all $\alpha \geq 0$, \, (ii)\;\;$C+C \subseteq C$ and \, (iii) \,$C\cap (-C)=\{0\}$. 
Cones play an important role in the theory of ordered vector spaces because a partial ordering in a vector space can be given in terms of the corresponding positive cone. For more information about cones and their properties we refer to \cite{cones}.  

Before we proceed, let us introduce some notation. If $V$ is a mixed lattice space then $V_p=\{x\in V:x\geq 0\}$ and $V_{sp}=\{x\in V:x\sgeq 0\}$ are the $(\leq)$-\emph{positive cone} and the $(\sleq)$-\emph{positive cone} of $V$, respectively. 
\blue{Accordingly, an element $x\in V_p$ is called \emph{positive}, and if $x\in V_{sp}$ then $x$ is said to be \emph{specifically positive}, or \emph{$(\sleq)$-positive}.} 
In mixed lattice spaces we are particularly interested in cones that are also mixed lattice semigroups.

\begin{defn}
A cone \blue{$C\subseteq V_{sp}$} in a mixed lattice space $V$ is called a \emph{mixed lattice cone} if $x\uenv y$ and $x\lenv y$ belong to $C$ whenever $x,y\in C$. 
\end{defn}

If $E$ is any subset of $V$ then we define $E_p=E\cap V_p$ and $E_{sp}=E\cap V_{sp}$. 
The notions of regular and quasi-regular mixed lattice space are defined similarly as in mixed lattice groups. 
Hence, using the notation just introduced, $V$ is \emph{quasi-regular} if $V_{sp}$, the positive cone associated with the specific order, is a mixed lattice cone.  
If $C$ is a cone in a vector space $V$ then $S=C-C$ is called the \emph{subspace generated by} $C$. A mixed lattice space $V$ is \emph{regular} if $V_{sp}$ is a generating mixed lattice cone, that is, $V=V_{sp}-V_{sp}$.

The rules (M1)--(M8) for the mixed envelopes 
remain valid in quasi-regular 
mixed lattice spaces. In a mixed lattice space we can add to the list the following rules concerning the scalar multiplication. If $V$ is a mixed lattice space and $0\leq a\in \R$ then the following hold for all $x,y\in V$

$$ 
\begin{array}{ll}
(M9)  &  (ax)\lenv(ay) = a(x\lenv y) \; \textrm{ and } \; (ax)\uenv(ay) = a(x\uenv y) \quad (a\geq 0). \bigskip \\ 
\end{array} 
$$

The notions of upper and lower parts of an element and the generalized asymmetrical absolute values were introduced in \cite{jj1}. 
If $x\in V$ then the elements $^{u}x=x\uenv 0$ and $^{l}x=(-x)\uenv 0$ are called the \emph{upper part} and \emph{lower part} of $x$, respectively. Similarly, the elements $x^{u}=0\uenv x$ and $x^{l}=0\uenv (-x)$ are called the \emph{specific upper part} and \emph{specific lower part} of $x$, respectively. These play a similar role as the positive and negative parts of an element in a Riesz space. The \emph{generalized absolute values} of $x$ are then defined as\, $\ul{x}= \lupp{\!x} + \rlow{x}$  and  $\lu{x}=\llow{\!x} + \rupp{x}$. The elements $\ul{x}$ and $\lu{x}$ are distinct, in general, and they are ''asymmetrical'' in the sense that $\ul{x} =\lu{(-x)}$ for all $x$.

The upper and lower parts and the generalized absolute values 
have several important basic properties, which were proved in \cite{jj1}. These properties are given in the next theorem.

\begin{thm}\label{absval}
Let $V$ be a quasi-regular 
mixed lattice space and $x\in V$. Then the following hold.
\begin{enumerate}[(a)]
\item
\; $\lupp{x}\,=\,\llow{(-x)}$ \quad \textrm{and} \quad $\rupp{x}\,=\,\rlow{(-x)}$. 
\item
\; $x \, = \, \rupp{x} \, -\,\llow{x} \, = \, \lupp{x} \, - \, \rlow{x}$.
\item
\;  $\ul{x} \, =\, \lupp{x}\uenv \rlow{x} \, =\, \lupp{x} \,+\,\rlow{x}$ \quad \textrm{and} \quad $\lu{x} \, =\,\llow{x}\uenv \rupp{x} \, =\, \llow{x} \,+\,\rupp{x}$.
\item
\; $\ul{x} \,=\,\lu{(-x)}$.
\item
\;
$\lupp{x}+\lupp{y} \geq \lupp{(x+y)}$, \quad $\rlow{x}+\rlow{y} \geq \rlow{(x+y)}$ \quad and \quad $\ul{x}+\ul{y} \geq \ul{(x+y)}$.
\item
\;
$\rupp{x}+\rupp{y} \geq \rupp{(x+y)}$, \quad $\llow{x}+\llow{y} \geq \llow{(x+y)}$ \quad and \quad $\lu{x}+\lu{y} \geq \lu{(x+y)}$.
\item
\; $\rupp{x}\lenv \llow{x} \, = \, 0\, = \, \rlow{x}\lenv \lupp{x}$.
\item
\; $\rupp{x}\uenv \llow{x}\, = \,\lupp{x}\,+\,\llow{x}\, = \,\rlow{x}\,+\,\rupp{x}\, = \,\rlow{x}\uenv \lupp{x}$
\item
\; $x\sgeq 0$ \, if and only if \, $x=\lu{x}=\ul{x}=\lupp{x}=\rupp{x}$\, and \,$\llow{x}=\rlow{x}=0$.
\item
\; $x\geq 0$ \, if and only if \, $x=\ul{x}=\lupp{x}$\, and \,$\rlow{x}=0$.
\item
\; $\ul{x}\geq 0$\,  and  \,$\lu{x}\geq 0$. Moreover,  $\ul{x}=\lu{x}= 0$\,  if and only if  $x=0$.
\item
\; $\ul{(ax)}=a\,\ul{x}$  \, and \, $\lu{(ax)}=a\,\lu{x}$ \, \textrm{for all} $a\geq 0$.
\item
\; $\ul{(ax)}=|a|\,\lu{x}$  \, and \, $\lu{(ax)}=|a|\,\ul{x}$ \, \textrm{for all} $a<0$.
\item
\; $2(x\lenv y) = x+y-\lu{(x-y)}$ \quad and \quad $2(y\lenv x) = x+y-\ul{(x-y)}$.
\end{enumerate}
\end{thm}

The next theorem was also proved in \cite{jj1}, but under more general assumptions. If $V$ is assumed to be quasi-regular then we have a sharper version of the same theorem, as given below. The proof is almost identical to the more general case, see \cite[Theorem 3.6]{jj1}.

\begin{thm}\label{disjoint}
Let $V$ be a quasi-regular mixed lattice space and $x\in V$. 
\begin{enumerate}[(a)]
\item 
If $x=u-v$ with $u\sgeq 0$ and $v\geq 0$, then $0\leq \rupp{x}\leq u$ and $0\leq v\leq \llow{x}$. If also $v\sgeq 0$ then $0\sleq \rupp{x}\sleq u$ and $0\leq \llow{x} \sleq v$.
\item 
Conversely, if $x=u-v$ with $u\lenv v=0$ then  $u= \rupp{x}$ and $v= \llow{x}$. 
\end{enumerate}
\end{thm}

\blue{
\emph{For the remainder of this paper, we shall always assume that $V$ is a quasi-regular mixed lattice 
vector space unless otherwise stated.} Due to this convention, we shall drop the term ''quasi-regular'', and henceforth, by a mixed lattice space we mean a quasi-regular mixed lattice space. 
}

The notions of mixed lattice subspaces and ideals 
were introduced in \cite{jj1}. 
A subspace $S$ of a mixed lattice vector space $V$ is called a \emph{mixed lattice subspace} of $V$ if $x\uenv y$ and $x\lenv y$ belong to $S$ whenever $x$ and $y$ are in $S$. 
A subset $U\subseteq V$ is called \emph{($\leq$)-order convex}, if $x\leq z\leq y$ and $x,y\in U$ imply that $z\in U$. Similarly, a subset $U\subseteq V$ is called \emph{($\sleq$)-order convex}, if $x\sleq z\sleq y$ and $x,y\in U$ imply that $z\in U$. 
A subspace $A$ is ($\leq$)-order convex if and only if $0\leq y\leq x$ and $x\in A$ imply that $y\in A$. Similarly, a subspace $A$ is ($\sleq$)-order convex if and only if $0\sleq y\sleq x$ and $x\in A$ imply that $y\in A$.
If $A$ is a ($\leq$)-order convex mixed lattice
subspace of $V$ then $A$ is called a \emph{mixed lattice ideal} of $V$. Similarly, a ($\sleq$)-order convex mixed lattice subspace of $V$ is called a \emph{specific mixed lattice ideal} of $V$.
For brevity, these will be called simply \emph{ideals} and \emph{specific ideals}, respectively. This should not cause any confusion as we are mostly dealing with mixed lattice ideals in this paper. If we refer to other type of ideals (such as lattice ideals) then we will emphasize it accordingly.

To study the relationships between ideals and specific ideals, as well as the structure of mixed lattice spaces in more detail, we introduce the following definitions.

\begin{defn}\label{def_proper_sp_ideal}
\begin{enumerate}[(i)]
\item
A subspace $S$ is called \emph{regular}, if $S=S_{sp}-S_{sp}$.
\item
A subspace $S$ is called \emph{positively generated}, if $S=S_p -S_p$. 
\item
A subspace $A$ is called \emph{mixed--order convex} 
if $y\in A$ and $0\sleq x\leq y$ together imply that $x\in A$.
\item
A mixed--order convex mixed lattice subspace is called a \emph{quasi-ideal}.
\item
A specific ideal $A$ is called a \emph{proper specific ideal} if $A$ is not a quasi-ideal. Similarly, a quasi-ideal $A$ is called a  \emph{proper quasi-ideal} if $A$ is not an ideal. 
\end{enumerate}
\end{defn}

Every mixed lattice subspace is positively generated.  
It is also easy to see that a mixed lattice subspace $S$ is regular if and only if for every $x\in S$ there exists some $z\in S_{sp}$ such that $x\sleq z$. 
There are ideals that are not regular (Example \ref{rdisjoint_esim}). In many cases, proper specific ideals and proper quasi-ideals are regular, but it is not known whether or not this is true in general.

Before we proceed, we should remark that if $A$ is a subspace of $V$ then, in order to show that $A$ is a mixed lattice subspace it is sufficient to show that $\rupp{x}\in A$ for every $x\in A$. Indeed, if this holds and $x,y\in A$ then $y\uenv x = y+ \rupp{(x-y)}\in A$. Consequently, $x\lenv y =x+y- y\uenv x \in A$. This observation simplifies many of the proofs that follow.

Now the following result holds.

\begin{prop}\label{regular_oconvex}
Every regular $(\sleq)$-order convex subspace is a mixed lattice subspace, and hence a specific ideal.
\end{prop}

\begin{bproof}
If $x\in A$ then $x=u-v$ with $u,v\in A_{sp}$. By Theorem \ref{disjoint} we have $0\sleq \rupp{x} \sleq u$ and so  
it follows that $\rupp{x}\in A$. Hence, $A$ is a regular $(\sleq)$-order convex mixed lattice subspace, by the remark made before the proposition.
\end{bproof}

The notion of a mixed-order convex subspace is intermediate between ($\sleq$)-order convex and ($\leq$)-order convex subspaces. Every ($\leq$)-order convex subspace is mixed-order convex, and every mixed-order convex subspace is ($\sleq$)-order convex. In particular, every ideal is a quasi-ideal, and every quasi-ideal is a specific ideal.

The following simple lemma is useful in the study of mixed-order convex subspaces.

\begin{lemma2}\label{mixed-o-convex}
A subspace $A$ is mixed-order convex if and only if $0\leq x\sleq y$ with $y\in A$ imply that $x\in A$.
\end{lemma2}

\begin{bproof} 
Assume that $0\sleq x\leq y$ with $y\in A$ implies that $x\in A$. We observe that if $0\leq u\sleq v$ with $v\in A$ then $0\sleq v-u\leq v$, and this implies that $v-u\in A$ by assumption. Then $u=v-(v-u)\in A$ since $A$ is a subspace. The converse implication is similar.
\end{bproof}


Next we will study some basic properties of mixed lattice cones.

\begin{thm}\label{mlcone1}
Let $V$ be a 
mixed lattice space and $C$ a mixed lattice cone in $V$. Then the subspace generated by $C$ is a mixed lattice subspace. 
Conversely, if $S$ is a mixed lattice subspace of $V$ then $S_{sp}$ is a mixed lattice cone in $V$.
\end{thm}

\begin{bproof}
Let $C$ be a mixed lattice cone. It is clear that $S=C-C$ is a subspace. 
If $x,y\in S$ then $x=u-v$ and $y=a-b$ for some $u,v,a,b\in C$. Then
$$
x\lenv y = (u-v)\lenv (a-b) = (u+b) \lenv (a+v) -(v+b),
$$
where $(u+b) \lenv (a+v)\in C$.  
Also $v+b\in C$, so $x\lenv y \in S$. Similarly $x\uenv y \in S$ and so $S$ is a mixed lattice subspace.

Conversely, if $S$ is a mixed lattice subspace then $S_{sp}$ is clearly a cone. If $x,y\in S_{sp}$ then $x\lenv y\in S$ and $x\lenv y\sgeq 0$, by (M8b), and so $x\lenv y\in S_{sp}$. 
Similarly, $x\uenv y\in S$, and $0\sleq x\sleq x\uenv y$, so $x\uenv y\in S_{sp}$. 
Hence $S_{sp}$ is a mixed lattice cone.
\end{bproof}

To gain more information about the structure of specific ideals, we introduce some additional terminology.  
A mixed lattice cone $C_1$ is called a \emph{mixed lattice sub-cone} of another mixed lattice cone $C_2$ if $C_1\subseteq C_2$. 
A sub-cone $F$ of a cone $C$ is called a \emph{face} of $C$ if 
$x + y\in F$ with $x,y\in C$ imply that $x,y\in F$.

\blue{First we observe that every face of $V_{sp}$ is a mixed lattice sub-cone.} 

\blue{\begin{prop}\label{face_is_mlcone}
Every face of $V_{sp}$ is a mixed lattice sub-cone of $V_{sp}$.
\end{prop}}

\blue{\begin{bproof}
Let $F$ be a face of $V_{sp}$ and $x,y\in F$. Then $x\uenv y + y\lenv x=x+y\in F$. It follows that $x\uenv y \in F$ and $y\lenv x \in F$, and so $F$ is a mixed lattice sub-cone of $V_{sp}$.
\end{bproof}}

\blue{The next two results are well-known but we give proofs for completeness.} 

\blue{\begin{lemma2}\label{o-convex_cone}
A sub-cone $C$ of $V_p$ is ($\leq$)-order convex if and only if $0\leq x\leq y$ with $y\in C$ implies that $x\in C$. Similarly, a sub-cone $C$ of $V_{sp}$ is ($\sleq$)-order convex if and only if $0\sleq x\sleq y$ with $y\in C$ implies that $x\in C$.
\end{lemma2}}

\blue{\begin{bproof}
Clearly the given condition is necessary. Assume the condition holds and let $z\leq x\leq y$ with $z,y\in C$. Then $0\leq x-z\leq y-z\leq y$ which implies that $x-z\in C$, and since $z\in C$ and $C$ is a sub-cone, it follows that $x=(x-z)+z\in C$. The proof of the second statement is essentially the same.
\end{bproof}}

\blue{\begin{prop}\label{facechar}
A sub-cone $F$ of $V_p$ is a face of $V_{p}$ if and only if $F$ is ($\leq$)-order convex. 
Similarly, $F$ is a face of $V_{sp}$ if and only if $F$ is a ($\sleq$)-order convex sub-cone of $V_{sp}$.
\end{prop}}

\blue{\begin{bproof}
Suppose that $F$ is a face of $V_p$ and let $0\leq x\leq y$ with $y\in F$. Then $x\in V_p$, $0\leq y-x\in V_p$ and $(y-x)+x=y\in F$. It follows that $x\in F$ and this shows that $F$ is ($\leq$)-order convex, by the preceding lemma. 
Conversely, if $F$ is ($\leq$)-order convex and $x+y\in F$ then $0\leq x\leq x+y$ and $0\leq y\leq x+y$ imply that $x,y\in F$. Hence $F$ is a face. The proof of the second assertion is identical, just replace $\leq$ by $\sleq$.
\end{bproof}}

\blue{The following result is now an immediate consequence of the preceding results and the definition of a specific ideal.}

\blue{\begin{cor}\label{spidealface}
A mixed lattice subspace $A$ is a specific ideal if and only if $A_{sp}$ is a face of  $V_{sp}$.
\end{cor}}

\blue{Our next result is important for the theory of ideals. It shows that every ideal contains a quasi-ideal.}

\begin{thm}\label{smallest_ideal_inside_ideal}
If $A$ is an ideal and $W=A_{sp}-A_{sp}$ then $W$ is a quasi-ideal, 
and there is no ideal $B$ such that $W\subseteq B\subseteq A$ and $B\neq A$. In particular, if $W$ is an ideal then $W=A$. 
\end{thm}

\begin{bproof}
By Theorem \ref{mlcone1}, $W$ is a mixed lattice subspace. Moreover, if $0\sleq x \leq y$ with $y\in W$ then $y\in A$, and $A$ is an ideal, so $x\in A_{sp}\subseteq W$. This shows that $W$ is mixed-order convex.  
Let $B$ be any ideal contained in $A$ such that $W\subseteq B$. Then for every $x\in A_p$ we have $0\leq x\leq \rupp{x}\in B$. This implies that $x\in B$, so $B=A$. In particular, if $W$ is an ideal then $W=A$.
\end{bproof}

We recall that if $E$ is a subset of $V$ then the smallest ideal (with respect to set inclusion) that contains $E$ is called the \emph{ideal generated by} $E$. 
The next result gives a description of ideals generated by a mixed lattice cone.

\begin{thm}\label{ideal_description}
The ideal $A$ generated by a mixed lattice cone $C$ equals the subspace generated by the cone $S=\{x\in V: 0\leq x\leq u \, \textrm{ for some } \,u\in C \}$. 
\end{thm}

\begin{bproof}
Since $C$ is a cone, it follows immediately that $S$ is also a cone.   
Evidently, every ideal that contains $C$ must contain $S$, so it contains also the subspace $S-S$.  
Therefore, it is sufficient to show that the subspace $S-S$ is an ideal. $S-S$ is obviously ($\leq$)-order convex, so it remains to show that it is a mixed lattice subspace. Let $x\in S-S$. Then $x=x_1-x_2$ with $x_1,x_2\in S$, and so $x\leq x_1\leq u_1$ for some $u_1\in C$. 
\blue{Since $C\subseteq V_{sp}$, it follows by Theorem \ref{absval}(i) that $\rupp{(u_1)}=u_1$,  
and so by (M5) we have 
$$
0\leq \rupp{x} =0\uenv x \leq 0\uenv x_1 \leq 0\uenv u_1 = \rupp{(u_1)}=u_1 \in C.
$$
} 
Thus, $\rupp{x}\in S$, and so $S-S$ is a mixed lattice subspace. Hence we have proved that $S-S$ is a smallest ideal containing $C$.  
\end{bproof}

If $A$ is an ideal then by Theorem \ref{mlcone1} the set $A_{sp}$ is a mixed lattice cone, and by Theorem \ref{smallest_ideal_inside_ideal} $A$ is the smallest ideal that contains $A_{sp}$. Now the two preceding theorems yield the following corollary which tells that ideals are uniquely determined by their sets of $(\sleq)$-positive elements.

\begin{cor}\label{ideals_corollary}
Every ideal $A$ equals the ideal generated by $A_{sp}$, that is, the subspace generated by the set $S=\{x\in V: 0\leq x\leq u \, \textrm{ for some } \,u\in A_{sp} \}$. \blue{In fact, $S=A_p$.} Consequently, if $A$ and $B$ are two ideals such that $A_{sp}=B_{sp}$ then $A=B$.
\end{cor}

\blue{\begin{bproof}
The only thing that still needs proof is the claim that $S=A_p$. The inclusion $S\subseteq A_p$ is clear, and the reverse inclusion holds too, for if $x\in A_p$ then $0\leq x\leq \rupp{x}\in A_{sp}$.
\end{bproof}}

The preceding results have yet another consequence. By Theorem \ref{smallest_ideal_inside_ideal}, if $A$ is an ideal then the subspace $W=A_{sp}-A_{sp}$ is a regular quasi-ideal. The following gives the converse.

\begin{thm}\label{quasi-ideal_charact}
If $W$ is a regular quasi-ideal then there exists an ideal $A$ such that $W=A_{sp}-A_{sp}$. Hence, $W$ is a regular quasi-ideal if and only if $W=A_{sp}-A_{sp}$ for some ideal $A$.
\end{thm}

\begin{bproof}
Let $W=W_{sp}-W_{sp}$ be a regular quasi-ideal and let $A$ be the ideal generated by $W_{sp}$. By Theorem \ref{mlcone1} the set $W_{sp}$ is a mixed lattice cone, so by Theorem \ref{ideal_description} $A=S-S$, where  $S=\{x\in V: 0\leq x\leq u \, \textrm{ for some } \,u\in W_{sp} \}$. Now we only need to show that $A_{sp}=W_{sp}$. It is clear that $W_{sp}\subseteq A_{sp}$, so let $x\in A_{sp}$. Then $x=y-z$ where $y,z\in S$, so we have $0\sleq x \leq y\leq u$ for some $u\in W_{sp}$. But this implies that $x\in W_{sp}$ since $W$ is a quasi-ideal, so we conclude that $A_{sp}=W_{sp}$.
\end{bproof}

Next we will study the properties of algebraic sums of different types of subspaces. For instance, if $A$ and $B$ are quasi-ideals, then what can be said about the sum $A+B$ ? Another important question is if the sets of positive elements are preserved in these sums, that is, does $(A+B)_{sp}=A_{sp}+B_{sp}$ or $(A+B)_{p}=A_{p}+B_{p}$ hold.

\begin{thm}\label{sum_of_ideals2}
If $A$ is a quasi-ideal and $B$ is a positively generated ($\leq$)-order convex subspace then $A+B$ is a mixed-order convex subspace. 
Moreover, if $B$ is an ideal then $A+B$ is a quasi-ideal, and if, in addition, $A$ is regular then $(A+B)_p=A_p+B_p$.
\end{thm}

\begin{bproof}
It is clear that $A+B$ is a subspace. 
Assume first that $0\sleq x\leq y$ with $y=y_1+y_2\in A+B$. By assumption, $y_2=v_1-v_2$ with $0\leq v_1,v_2\in B$, and $A$ is a mixed lattice subspace, so $\rupp{(y_1)}\in A$. Then $0\sleq x\leq y_1+y_2 \leq \rupp{(y_1)} + v_1$, so we can now apply the dominated decomposition property (Theorem \ref{riesz_dec}(a)) and write $x=x_1+x_2$, where $0\sleq x_1 \leq \rupp{(y_1)}$ and $0\leq x_2 \leq v_1$. It follows that $x_1\in A$ and $x_2\in B$, so $x\in A+B$. 
This shows that $A+B$ is mixed-order convex. 

Assume next that $B$ is an ideal and let $y=y_1+y_2\in A+B$. Then $0\sleq \rupp{y} \leq  \rupp{(y_1)} + \rupp{(y_2)}$ and, applying Theorem \ref{riesz_dec}(a) again we can find elements $u_1$ and $u_2$ such that $\rupp{y}=u_1+u_2$, \, $0\sleq u_1 \leq \rupp{(y_1)}$ \, and \, $0\sleq u_2 \leq \rupp{(y_2)}$. This implies that $u_1\in A$ and $u_2\in B$, and hence 
 $\rupp{y}\in A+B$. This shows that $A+B$ is a mixed lattice subspace, and hence a quasi-ideal.
It is clear that $A_p+B_p\subseteq (A+B)_p$. Conversely, if $0\leq x\in A+B$ then $x=x_1+x_2$ with $x_1\in A$ and $x_2\in B$. If $A$ is regular, we can choose an element $0\sleq v\in A$ such that $x_1\sleq v$. Also, $x_2\sleq \lupp{x_2}\in B$, and so we have $0\leq x\sleq v+\lupp{x_2}$. We can now apply Theorem \ref{riesz_dec}(b) to find positive elements $u_1\in A$ and $u_2\in B$ such that $x=u_1+u_2$, completing the proof.
\end{bproof}

Along the same lines we have the following result.

\begin{thm}\label{specific_direct_sum1}
In a mixed lattice space the following hold. 
\begin{enumerate}[(a)]
\item
If $A$ and $B$ are quasi-ideals then $(A + B)_{sp}=A_{sp}+B_{sp}$. Moreover, if $A$ and $B$ are regular then $A + B$ is also a regular quasi-ideal. 
\blue{\item 
If $A$ is a regular specific ideal and $B$ is a regular quasi-ideal then $A + B$ is a regular specific ideal and $(A + B)_{sp}=A_{sp}+B_{sp}$.  
The last equality holds, in particular, if $B$ is an ideal.}
\end{enumerate} 
\end{thm}

\begin{bproof}
\emph{(a)} \, If $x\in (A + B)_{sp}$ then $0\sleq x = x_1+x_2\leq \rupp{(x_1)}+\rupp{(x_2)}$ with $\rupp{(x_1)}\in A_{sp}$ and $\rupp{(x_2)}\in B_{sp}$. By Theorem \ref{riesz_dec}(a)  $x=a+b$ for some elements $a$ and $b$ such that $0\sleq a\leq \rupp{(x_1)}$ and $0\sleq b\leq \rupp{(x_2)}$. It follows that $a\in A_{sp}$ and $b\in B_{sp}$, hence $(A+ B)_{sp}\subseteq A_{sp}+B_{sp}$. The reverse inclusion is obvious. Assume then that $0\sleq y \leq x$ with $x\in A+ B$. If $A$ and $B$ are regular it follows that $A + B$ is also regular, and we may thus assume that $x\sgeq 0$. Hence, $0\sleq y \leq x=a+b$ where $a\in A_{sp}$ and $b\in B_{sp}$. Now the same argument as above (using Theorem \ref{riesz_dec}(a)) shows that $y\in A + B$, and so $A + B$ is mixed-order convex. Then 
$A + B$ is a mixed lattice subspace, by Proposition \ref{regular_oconvex}.

\blue{ 
\emph{(b)} \, If $x\in (A + B)_{sp}$ then $x = x_1+x_2$ with $x_1\in A$ and $x_2\in B$. Since $A$ and $B$ are regular, there exist elements $u\in A_{sp}$ and $v\in B_{sp}$ such that $x_1\sleq u$ and $x_2\sleq v$. Hence $0\sleq x = x_1+x_2\sleq u+v$, and 
by Theorem \ref{riesz_dec}(a) we have $x=a+b$ for some elements $a$ and $b$ such that $0\sleq a\sleq u$ and $0\sleq b\leq v$. It follows that $a\in A_{sp}$ and $b\in B_{sp}$, hence $(A+ B)_{sp}\subseteq A_{sp}+B_{sp}$. The reverse inclusion is clear. In particular, if $B$ is an ideal then $W=B_{sp}-B_{sp}$ is a regular quasi-ideal, and so $V_{sp}= A_{sp}+W_{sp}=A_{sp}+B_{sp}$. The proof that $A + B$ is a regular specific ideal is similar to part (a).} 
\end{bproof}

Regarding the above theorem we note that by Theorem \ref{quasi-ideal_charact} regular quasi-ideals are precisely those subspaces $W$ that   $W=A_{sp}-A_{sp}$ for some ideal $A$. As with ideals in Riesz spaces, these regular quasi-ideals form a distributive lattice.

\begin{thm}\label{distrib_lattice}
Let $V$ be a mixed lattice space and denote by $\mathcal{R}(V)$ the set of all regular quasi-ideals of $V$, ordered by inclusion. Then  $\mathcal{R}(V)$ is a distributive lattice where $A\vee B=A+B$ and $A\wedge B=A\cap B$. Moreover, $\mathcal{R}(V)$ has the smallest element $\{0\}$ and the largest element $V_{sp}-V_{sp}$.
\end{thm}

\begin{bproof}
If $A,B\in \mathcal{R}(V)$ then by Theorem \ref{specific_direct_sum1} we have $A+B\in \mathcal{R}(V)$, and $A+B$ is clearly the smallest regular quasi-ideal that contains both $A$ and $B$, hence $A+B=A\vee B$. It is also clear that $A\cap B$ is the largest quasi-ideal that is contained in both $A$ and $B$. To see that $A\cap B$ is regular, let $x\in A\cap B$. Since $A$ and $B$ are regular, there exist elements $u\in A_{sp}$ and $v\in B_{sp}$ such that $x\sleq u$ and $x\sleq v$ hold. Then $x\sleq u\lenv v$ and the inequalities $0\sleq u\lenv v \sleq u$ and $0\sleq u\lenv v \leq v$ imply that $u\lenv v \in A\cap B$, proving that $A\cap B$ is regular, and hence $A\cap B=A\wedge B$. 
To prove the distributivity, it is sufficient to show that $[(A\cap B)+(A\cap C)]_{sp}= [A\cap(B+C)]_{sp}$, since $(A\cap B)+(A\cap C)$ and $A\cap(B+C)$ are both regular quasi-ideals. 
The inclusion $(A\cap B)+(A\cap C)\subseteq A\cap(B+C)$ is rather trivial, so let $x\in [A\cap(B+C)]_{sp}$. 
Then $x\in (B+C)_{sp}$, so  $x=x_1+x_2$ where $x_1\in B_{sp}$ and $x_2\in C_{sp}$, by Theorem \ref{specific_direct_sum1}. Moreover, $0\sleq x_1 \sleq x \in A$ implies that $x_1\in A_{sp}$, and similarly, $x_2\in A_{sp}$. Hence, $x_1\in (A\cap B)_{sp}$ and $x_2\in (A\cap C)_{sp}$, and therefore $x\in [(A\cap B)+(A\cap C)]_{sp}$. This shows that $\mathcal{R}(V)$ is distributive. Finally, it is clear that $\{0\}$ is the smallest element in $\mathcal{R}(V)$, and by Theorem \ref{smallest_ideal_inside_ideal} the largest element is  $W=V_{sp}-V_{sp}$.
\end{bproof}

\blue{We do not know if the preceding theorem can be stated for regular specific ideals in general. However, in many cases, a mixed lattice space is also a lattice with respect to one (or both) partial orderings. In this case, we obtain stronger results as all the lattice-theoretic tools are available to us. We will now briefly consider this situation. To avoid any confusion, we need to fix some terminology. 
If $(V,\sleq)$ is a vector lattice then we will say that $V$ is a $(\sleq)$-\emph{lattice}. 
If $(V,\sleq)$ is a vector lattice and $A$ is a lattice ideal in $(V,\sleq)$, then $A$ is called a $(\sleq)$-\emph{lattice ideal}.  
We also denote the absolute value and the positive and negative parts of an element $x$ with respect to $\sleq$ by $\spec |x|$,  $\spec (x^+)$ and $\spec (x^-)$, respectively.}

\blue{\begin{thm}\label{lattice_ideals}
If $V$ is a mixed lattice space such that $(V,\sleq)$ is a vector lattice then $A$ is a regular specific  
ideal in $V$ if and only if $A$ is a $(\sleq)$-lattice ideal of $(V,\sleq)$. 
\end{thm}}

\blue{\begin{bproof}
It was proved in \cite[Proposition 4.6]{jj1} that every $(\sleq)$-lattice ideal is a specific ideal. Moreover, a $(\sleq)$-lattice ideal $A$ is regular (since every $x\in A$ can be written as $x=\spec (x^+)-\spec (x^-)$ where $\spec (x^+),\spec (x^+)\in A_{sp}$), so we only need to prove the converse. Let $A$ be a regular specific ideal. Then $A$ is $(\sleq)$-order convex and if $x\in A$ then also $\ul{x}\in A$ and $\lu{x}\in A$. Since $A$ is regular, there is an element $u\in A_{sp}$ such that $\ul{x}\sleq u$ and $\lu{x}\sleq u$.  
By \cite[Proposition 3.16]{jj1} the absolute value of $x$ formed with respect to $\sleq$ is given by 
$\spec |x| =\spsup \{\ul{x},\lu{x}\}$. 
Thus we have $0\sleq \spec |x| =\spsup \{\ul{x},\lu{x} \} \sleq u \in A$. Since $A$ is $(\sleq)$-order convex, it follows that $\spec |x|\in A$. This shows that $A$ is a $(\sleq)$-lattice ideal. 
\end{bproof}}

\blue{It is well known that the set of lattice ideals in a Riesz space is a distributive lattice. Also, every regular quasi-ideal is a regular specific ideal, so putting all this together with Theorems \ref{lattice_ideals} and \ref{distrib_lattice} we obtain the following:}

\blue{\begin{cor}
Let $V$ be a mixed lattice space that is a lattice with respect to $\sleq$, and denote by $\mathcal{L}(V)$ the set of all regular specific ideals of $V$, ordered by inclusion. Then  $\mathcal{L}(V)$ is a distributive lattice where $A\vee B=A+B$ and $A\wedge B=A\cap B$. Moreover, $\mathcal{L}(V)$ has the smallest element $\{0\}$ and the largest element $V$, and the set $\mathcal{R}(V)$ of all regular quasi-ideals of $V$ is a sub-lattice of $\mathcal{L}(V)$.
\end{cor}}


\blue{
Now we turn to the discussion of bands in mixed lattice spaces.

\begin{defn}\label{bands}
Let $V$ be a mixed lattice space. 
A specific ideal $A$ is called a \emph{specific band} if $\spsup E \in A$ whenever $E$ is a non-empty subset of $A$ such that $\spsup E$ exists in $V$. 
If $A$ is a quasi-ideal with the above property then $A$ is called a \emph{quasi-band}. 
An ideal $B$ is called a \emph{band} if $\sup E \in A$ whenever $E$ is a non-empty subset of $A$ such that $\sup E$ exists in $V$. 
\end{defn}

It follows from the identity $\inf E = -\sup (-E)$ that if $A$ is a band and $E$ is a non-empty subset of $A$ such that $\inf E$ exists in $V$ then $\inf E \in A$. Similarly, if $A$ is a specific band and $\spinf E$ exists in $V$ then $\spinf E \in A$. 

It is also clear that every 
quasi-band is a specific band. For the sequel, we need to introduce the following notions.

\begin{defn}
A specific ideal $A$ is called a \emph{weak specific band} if $\strsup E \in A$ whenever $E$ is a non-empty subset of $A$ such that $\strsup E$ exists in $V$. 
A quasi-ideal with the above property is called a \emph{weak quasi-band}, and an ideal with the above property is called a \emph{weak band}. 
\end{defn}

Weak bands and weak specific bands in mixed lattice semigroups were introduced by Arsove and Leutwiler \cite{ars}. 

Clearly, every specific band is a weak specific band, every quasi-band is a weak quasi-band and every band is a weak band. 
Now we can state the following characterization for weak bands and weak specific bands.

\begin{lemma2}\label{weakband_lemma}
If $A$ is a (specific) ideal 
in $V$ then the following are equivalent.
\begin{enumerate}[(a)]
\item
$A$ is a weak (specific) band. 
\item
$\strsup E\in A$ whenever $E$ is a non-empty subset of $A_p$ such that $\strsup E$ exists in $V$.
\item
$\strsup E\in A$ whenever $E$ is a non-empty subset of $A_{sp}$ such that $\strsup E$ exists in $V$.
\end{enumerate}
\end{lemma2}

\begin{bproof}
The implication $(a)\implies (b)$ is clear. Condition $(b)$ obviously implies $(c)$, since $A_{sp}\subseteq A_p$. 
Assume that $(c)$ holds and let $E$ be a non-empty subset of $A$ such that $u_0=\strsup E$ exists in $V$. Fix an element $x\in E$ and define $D=\{x\uenv u -x: u\in E\}$. Then $D$ is a non-empty subset of $A_{sp}$ and since $u_0\sgeq x$, by property (M6a) we have $x\uenv u_0=u_0$. Using Proposition \ref{strongsup_properties} we then have $\strsup D=\strsup\{x\uenv u -x:u\in E\}=x\uenv u_0 -x=u_0 -x\in A$, and hence $u_0=x+(u_0 -x)\in A$. 
\end{bproof}
}

\section{Disjoint complements}
\label{sec:s5}

If $x,y\in V$ and $x\lenv y=0$ then $x$ is said to be \emph{left-disjoint} with $y$ and $y$ is \emph{right-disjoint} with $x$. The reason for this terminology is, of course, the fact that in general $x\lenv y\neq y\lenv x$.
Next we will investigate the sets of those elements that are left or right disjoint with each element of a given subspace $A$.  
It follows immediately from the inequalities $x\sgeq x\lenv y$ and $y\geq x\lenv y$ that if $x\lenv y=0$ then we must have $x\sgeq 0$ and $y\geq 0$. Because of this, we will first consider disjointness for positive elements only. As it turns out, these sets of positive disjoint elements are in fact mixed lattice cones. We then define the left and right disjoint complements as the specific ideal and the ideal generated by these cones.

We begin with the left disjoint complement.

\begin{lemma2}\label{disjoint_sum_left}
If $x,y\sgeq 0$ and $z\geq 0$ with $y\lenv z=0$ then $(x+y)\lenv z = x\lenv z$. In particular, if also $x\lenv z=0$ then $(x+y)\lenv z=0$.
\end{lemma2}

\begin{bproof}
If  $y\lenv z=0$ and $x\sgeq 0$ then, using Theorem \ref{new_ineq} we obtain
$$
(x+y)\lenv z \sleq (x+y)\lenv (x+z)=x +y\lenv z=x.
$$  
On the other hand, $(x +y)\lenv z\leq z$, so we have $(x+y)\lenv z \leq x\lenv z$. The reverse inequality $0\leq x\lenv z \leq (x+y)\lenv z $ holds by (M5), and the lemma is proved. 
\end{bproof}

\blue{
We now introduce the set $\apld=\{x\in V: x\lenv z=0 \; \textrm{ for all } \; z\in A_p \}$ and study its properties.
}

\begin{thm}\label{pos_disj}
Let $V$ be a mixed lattice space and $A$ a mixed lattice subspace of $V$. 
Then the set 
$\apld$ defined above is a mixed lattice cone in $V$.
\end{thm}

\begin{bproof}
Let $x,y\in \apld$ and $z\in A_p$. Then $(x+y)\lenv z=0$ by the preceding lemma, and hence $x+y\in \apld$. 
If $0\leq \alpha\in\R$ and we put $c=\max\{\alpha, 1\}$, then $0\leq (\alpha x)\lenv z\leq  c(x\lenv z)=0$. Thus, $\alpha x \in \apld$. 
This shows that $\apld$ is a cone in $V$. 
Next we note that if $v=x\lenv y$ then  
$0\sleq v\sleq x$ and it follows that for every $z\in A_p$ we have  $0\leq v\lenv z\leq x\lenv z=0$, and thus $v\lenv z=0$. This shows that $v\in \apld$. If we set $w=x-x\lenv y$ then $0\sleq w\sleq x$, which implies that $0\leq w\lenv z \leq x\lenv z =0$ for all $z\in A_p$. Thus $w\in \apld$, and since $\apld$ is a cone, we have $y+w=y+ x-x\lenv y=y\uenv x \in \apld$. Hence, $\apld$ is a mixed lattice cone in $V$. 
\end{bproof}

The preceding result motivates the following definition.

\begin{defn}\label{ldc}
Let $A$ a mixed lattice 
subspace of a mixed lattice space $V$. The \emph{left disjoint complement} of $A$ is the specific ideal ${}^{\perp}A$ generated by the cone $\apld=\{x\sgeq 0: x\lenv z=0 \; \textrm{ for all } \; z\in A_p \}$.
\end{defn}

\begin{remark2}\label{remark_leftdisjcompl}
We should point out that, more generally, if $E$ is any subset of $V$ such that $E_p$ is non-empty, then ${}^{\perp}E_p$ is a mixed lattice cone in $V$. However, for the purposes of the present paper it is sufficient to restrict ourselves to mixed lattice subspaces. By doing so we can avoid some unnecessary complications that arise if one considers non-trivial subspaces $S$ such that $S_p=\{0\}$. 
For instance, in such cases the algebraic sum of a subspace and its disjoint complement would not be a direct sum, in general. We will discuss these matters further at the end of this paper. 
\end{remark2}

\blue{
\begin{thm}\label{left_disjoint_comp_is_band}
If $E$ is a subset of $({}^{\perp}A)_{sp}$ such that 
$\spsup E$ exists in $V$ then $\spsup E \in ({}^{\perp}A)_{sp}$. 
In particular,  
${}^{\perp}A$ is a regular weak specific band in $V$. Moreover, if ${}^{\perp}A$ is an ideal, then it is a weak regular band.
\end{thm}
}

\begin{bproof}
We will first show that ${}^{\perp}A$ is regular. 
Let $W = \apld - \apld$. It is clear that $W$ is a subspace and $W\subseteq {}^{\perp}A$, so we only need to show that $W$ is a specific ideal. 
It follows from Theorem \ref{pos_disj} and Theorem \ref{mlcone1} that $W$ 
is a mixed lattice subspace 
in $V$. Moreover, if $0\sleq y\sleq x$ with $x\in W$ then for every $z\in A_p$ we have  $0\leq y\lenv z\leq x\lenv z=0$, so $y\lenv z=0$ and thus $y\in W$. Hence, $W$ is a regular specific ideal,  
and so $W = {}^{\perp}A$. 

\blue{
Let $E$ be a non-empty 
subset of $({}^{\perp}A)_{sp}$ such that $u_0=\spsup E$ exists in $V$. Then, using (M1) and Theorem \ref{new_ineq} we have 
$$
u-u\lenv z = z\uenv u -z \sleq z\uenv u_0 -z = u_0 - u_0 \lenv z
$$
for all $u\in E$ and $z\in A_p$. But $u\lenv z=0$, so the above inequality reduces to $u\sleq u_0 - u_0 \lenv z$. Thus the element $u_0 - u_0 \lenv z$ is a $(\sleq)$-upper bound of the set $E$, so we have $u_0\sleq u_0 - u_0 \lenv z$. This implies that $0\leq u_0 \lenv z \sleq 0$ (where the inequality $0\leq u_0 \lenv z$ follows by (M5), since $u_0\sgeq 0$ and $z\geq 0$), and so $u_0 \lenv z=0$. Hence, $u_0 \in ({}^{\perp}A)_{sp}$. 
In particular, if $u_0$ is the strong supremum then ${}^{\perp}A$ is a weak specific band, and   
if ${}^{\perp}A$ is an ideal then 
${}^{\perp}A$ is a weak band, by Lemma \ref{weakband_lemma}.}
\end{bproof}

Next we turn to the right disjoint complement which we define in a similar manner as the left disjoint complement, but the situation is slightly more complicated. 
Let $A$ be a mixed lattice subspace and consider the set $\Ss(A)=\{x\geq 0:z\lenv x=0 \; \textrm{ for all } \; z\in A_{sp} \}$. 
It is easy to see that $\Ss(A)$ is closed under multiplication by positive scalars, \blue{by using a similar argument as in the proof of Theorem \ref{pos_disj}}. 
However, in general, $\Ss(A)$ is not closed under addition (see Example \ref{rdisjoint_esim}).

To get a better understanding of the situation, let us briefly examine the set $\Ss(A)$ more closely. 
Let $\mathcal{B}$ be the family of all those subsets of $\Ss(A)$ that are closed under addition. Notice that $\mathcal{B}$ is non-empty since $\{0\}\in \mathcal{B}$. Let $\mathcal{B}$ be ordered by inclusion. If $\mathcal{C}$ is a totally ordered subset of $\mathcal{B}$ then $\bigcup \{C:C\in \mathcal{C}\}$ is an upper bound of $\mathcal{C}$, and so  
by Zorn's lemma $\mathcal{B}$ has maximal elements. Let us denote by $\mathcal{M}(A)$ the set of maximal elements of $\mathcal{B}$, that is, the set of those subsets of $\Ss(A)$ that are maximal with respect to the property of being closed under addition. 
If $\Ss(A)$ itself is closed under addition then $\Ss(A)$ is the only element of $\mathcal{M}(A)$. Clearly, $\Ss(A)=\bigcup\{C:C\in \mathcal{M}(A)\}$, and each set in $\mathcal{M}(A)$ is a $(\leq)$-order convex cone. 

\blue{
\begin{prop}\label{subspace_components}
The set $\Ss(A)$ is $(\leq)$-order convex, and each $C\in \mathcal{M}(A)$ is a $(\leq)$-order convex cone.  
\end{prop}

\begin{bproof}
We divide the proof into 5 steps. 

(1)\, If $0\leq x\leq y$ with $y\in \Ss(A)$ then $0\leq z\lenv x\leq z\lenv y=0$ for all $z\in A_{sp}$, and so $z\lenv x=0$ for all $z\in A_{sp}$, proving that $\Ss(A)$ is $(\leq)$-order convex. 

(2)\, Let $C\in \mathcal{M}(A)$. Then $C$ is closed under addition by definition.

(3)\, If $0\leq x \leq y$ with $y\in C$, then $x\in \Ss(A)$ by step (1). Now, if $x\notin C$ and $x+z\in \Ss(A)$ for all $z\in C$ then $0\leq ax+bz\leq (a+b)(x+z) \in \Ss(A)$ for all $z\in C$ and $a,b\in \R_+$. By step (1) this implies that the set $D=\{ax+bz: z\in C, \, a,b\in \R_+ \}$ is contained in $\Ss(A)$. Clearly, $D$ is closed under addition, and $C$ is contained in $D$, contradicting the maximality of $C$. Hence, there exists some $z\in C$ such that $x+z\notin \Ss(A)$. But then $0\leq x+z\leq y+z\in C$, and by step (1) this implies that $x+z\in \Ss(A)$, a contradiction. Thus $x\in C$, and $C$ is $(\leq)$-order convex.

(4)\, If $x\in C$ then an inductive argument applied to step (2) shows that $nx\in C$ for all $n\in \N$. If $a\in \R_+$ then we can find some $m\in \N$ such that $a\leq m$, and so $0\leq ax\leq mx\in C$. By step (3) this implies that $ax\in C$, showing that $C$ is closed with respect to multiplication by positive scalars.

(5)\, Finally, $C\cap (-C)=\{0\}$, because the elements of $C$ are positive and $0\in C$. By steps (2)--(5), $C$ is a $(\leq)$-order convex cone. 
\end{bproof}
}

Due to the above result, we will call $\mathcal{M}(A)$ the \emph{set of maximal right disjoint cones of} $A$. It was noted above that, in general, the set $\Ss(A)$ is not closed under addition. In this sense, the set $\Ss(A)$ is ''too large''. However, we have the following result which is the analogue of Lemma \ref{disjoint_sum_left}.

\begin{lemma2}\label{disjoint_sum_lemma}
If $z,x\sgeq 0$ and $y\geq 0$ with $z\lenv y=0$ then $z\lenv (x+y) = z\lenv x$. In particular, if also $z\lenv x=0$ then $z\lenv (x+y)=0$.
\end{lemma2}

\begin{bproof}
Let $v=z\lenv (x+y)$. Then $v\leq x+y$, so $v-x\leq y$. On the other hand, since $x\sgeq 0$ we have $v-x\sleq v\sleq z$. Hence, $v-x\leq z\lenv y=0$, and so $v\leq x$. Since we also have $v\sleq z$, it follows that $0\leq v\leq z\lenv x$. The reverse inequality $0\leq z\lenv x \leq z\lenv (x+y)$ holds by (M5), and the desired result follows.
\end{bproof}

\blue{
It is clear that the set  $\aspp=\{x\sgeq 0:z\lenv x=0 \; \textrm{ for all } \; z\in A_{sp} \}$ is contained in $\Ss(A)$. The set $\aspp$ is in fact a mixed lattice cone. 
}

\begin{thm}\label{rightdisjoint}
If $A$ is a mixed lattice subspace of $V$ then the set $\aspp$  
defined above is a mixed lattice cone in $V$. 
\end{thm}

\begin{bproof}
The proof of the fact that $\aspp$ is a mixed lattice cone is similar to the proof of Theorem \ref{pos_disj}. The only real difference is in showing that the set $\aspp$ is closed under addition, and this follows immediately from the preceding lemma. 
\end{bproof}

\begin{defn}\label{rdc}
Let $V$ be a mixed lattice space and $A$ a mixed lattice subspace of $V$. The \emph{right disjoint complement} of $A$ is the ideal ${A}^{\perp}$ generated by the cone $\aspp=\{x\sgeq 0:z\lenv x=0 \; \textrm{ for all } \; z\in A_{sp} \}$.
\end{defn}

\begin{remark2}
In the above definition we require $A$ to be a mixed lattice subspace for similar reasons that were explained in Remark \ref{remark_leftdisjcompl} considering the left disjoint complement. The assumption that $A$ is a mixed lattice subspace guarantees that $A$ contains non-zero specifically positive elements (except, of course, in the trivial case $A=\{0\}$). 
\end{remark2}

\begin{thm}\label{right_disjoint_ideal}
$A^{\perp}$ is a weak band and 
$z\lenv x=0$ for all $0\leq x\in A^{\perp}$ and $z\in A_{sp}$. 
\end{thm}

\begin{bproof}
\blue{We will first show that $\aspp=\{w\in {A}^{\perp}: w\sgeq 0\}$, that is, the ideal ${A}^{\perp}$ does not contain any $(\sleq)$-positive elements that are not in $\aspp$. This will also justify the notation used for the set $\aspp$ in Definition \ref{rdc}. For this, 
let $W=\aspp -\aspp$ and $0\sleq x\leq y\in W$. Then $y=u-v$ for some $u,v\in \aspp$,  
and so $0\sleq x\leq u\in \aspp$. This implies that $0\sleq z\lenv x \leq z\lenv u=0$ for all $z\in A_{sp}$, and therefore $z\lenv x=0$, so $x\in \aspp\subseteq W$. This shows that $W$ is a regular quasi-ideal, and so by Corollary \ref{ideals_corollary} and Theorem \ref{quasi-ideal_charact}, $W_{sp}=\aspp=\{w\in {A}^{\perp}: w\sgeq 0\}$.

Now, if $z\in A_{sp}$ and  $0\leq x\in A^{\perp}$ then $0\leq x\leq \rupp{x}\in \aspp$, so $0\leq z\lenv x\leq z\lenv \rupp{x}=0$, which implies that $z\lenv x=0$. Next, let $E\subseteq (A^{\perp})_{sp}$ be a 
non-empty set and assume that $x_0=\strsup E$ exists in $V$. Then $x_0\sgeq 0$, and  
since $z\lenv x =0$ for all $x\in E$ and $z\in A_{sp}$, it follows by Proposition \ref{strongsup_properties}  that $z\lenv x_0 = \strsup_{x\in E} (z\lenv x)=0$, and hence $x_0\in (A^{\perp})_{sp}$. 
By Lemma \ref{weakband_lemma} this shows that $A^{\perp}$ is a weak band.}
\end{bproof}

The next result provides more information about the relationship between $A^{\perp}$ and $\Ss(A)$.

\begin{thm}\label{right_disjoint_ideal2}
If $A$ is a specific ideal 
then $(A^{\perp})_p \subseteq \bigcap \{C:C\in \mathcal{M}(A)\}$. 
\end{thm}

\begin{bproof}
We will first show that that $\aspp \subseteq \bigcap \{C:C\in \mathcal{M}(A)\}$. 
Let $A$ be a specific ideal and $C\in \mathcal{M}(A)$ with $x\in C$. Then it follows by Lemma \ref{disjoint_sum_lemma} that for every $w\in A_{sp}$ and $z\in \aspp$ we have 
$w\lenv(x+z)=0$, and so $x+z\in \Ss(A)$. This shows that $\aspp+C\subseteq \Ss(A)$. But $C$ is, by definition, a maximal cone in $\Ss(A)$, so we must have $\aspp\subseteq C$. This proves that $\aspp \subseteq \bigcap \{C:C\in \mathcal{M}(A)\}$.

Next, if  
$0\leq y\in A^{\perp}$ then $0\sleq \rupp{y}\in \aspp$. 
Thus, by what was just proved above, we have $\rupp{y}\in \bigcap \{C:C\in \mathcal{M}(A)\}$, and the inequality $0\leq y\leq \rupp{y}$ then implies that $y\in \bigcap \{C:C\in \mathcal{M}(A)\}$, by Proposition \ref{subspace_components}. Hence, $(A^{\perp})_p \subseteq \bigcap \{C:C\in \mathcal{M}(A)\}$.
\end{bproof}

In the next theorem we collect some basic properties of the 
disjoint complements. Some of them are straightforward consequences of the definitions.

\begin{thm}\label{properties_of_complements}
If $A$ is a mixed lattice subspace then the following hold.
\begin{enumerate}[(a)]
\item
$A \cap {}^{\perp}A=\{0\}$ \blue{and $A \cap A^{\perp}=\{0\}$}. 
\item
$A\subseteq ({}^{\perp}A)^{\perp}$. 
\item
If $A$ is regular then $A\subseteq {}^{\perp}(A^{\perp})$. 
\item
If $A$ is a quasi-ideal then $({A}^{\perp})_{sp}\subseteq ({}^{\perp}A)_{sp}$. 
\item
If $A$ is a quasi-ideal and ${}^{\perp}A$ is quasi-ideal then $({A}^{\perp})_{sp} = ({}^{\perp}A)_{sp}$. Moreover, if $A$ is a quasi-ideal  and ${}^{\perp}A$ is an ideal then ${}^{\perp}A={A}^{\perp}$.
\end{enumerate}
\end{thm}

\begin{bproof}
\begin{enumerate}[(a)]
\item
If $x\in A \cap {}^{\perp}A$ then also $\rupp{x},\rlow{x}\in A \cap {}^{\perp}A$, since $A$ and ${}^{\perp}A$ are mixed lattice subspaces. It follows that $\rupp{x}=\rupp{x}\lenv \rupp{x} =0$ and so $x\leq 0$. Similarly $\rlow{x}=0$, which implies that $x\geq 0$. Hence $x=0$. \blue{By a similar argument we have $A \cap A^{\perp}=\{0\}$.}
\item
This follows from the observation that $A_p \subseteq (({}^{\perp}A)^{\perp})_{p}$. \blue{Indeed, 
it follows from the definitions that $A_{sp} \subseteq (({}^{\perp}A)^{\perp})_{sp}$. 
Hence, if $x\in A_{p}$, we have $\rupp{x}\in A_{sp}\subseteq (({}^{\perp}A)^{\perp})_{sp}$. Since $({}^{\perp}A)^{\perp}$ is an ideal, it follows from $0\leq x\leq \rupp{x}$ that $x\in (({}^{\perp}A)^{\perp})_{p}$, and so $A_p \subseteq (({}^{\perp}A)^{\perp})_{p}$.} 
\item
It is evident that $A_{sp} \subseteq {}^{\perp}\aspp$ and if $A$ is regular then $A=A_{sp}-A_{sp} \subseteq {}^{\perp}\aspp - {}^{\perp}\aspp= {}^{\perp}(A^{\perp})$. 
\item
Let $x\in ({A}^{\perp})_{sp}$ and $z\in A_{sp}$. Then 
the inequalities $0\sleq x\lenv z \sleq x$ and $0\sleq x\lenv z \leq z$ imply  
that $x\lenv z \in ({A}^{\perp})_{sp}\cap A_{sp}=\{0\}$, and so $x\in ({}^{\perp}A)_{sp}$.
\item
If $x\in ({}^{\perp}A)_{sp}$ and $z\in A_{sp}$ then it follows from $0\sleq z\lenv x\sleq z$ and $0\sleq z\lenv x\leq x$ that $z\lenv x\in A\cap {}^{\perp}A=\{0\}$. Hence, $x\in \aspp$ and so $({}^{\perp}A)_{sp}\subseteq \aspp$. The reverse inclusion follows from (e). 
Then ${}^{\perp}A=({}^{\perp}A)_{sp}-({}^{\perp}A)_{sp}=\aspp-\aspp$, by Theorem \ref{left_disjoint_comp_is_band}, so if ${}^{\perp}A$ is an ideal then the equality ${}^{\perp}A={A}^{\perp}$ follows from Theorem \ref{smallest_ideal_inside_ideal}.
\end{enumerate}
\end{bproof}

\begin{remark2}
We note that, 
in particular, $\{0\}^{\perp}=V$ and ${}^{\perp} V=\{0\}$. Moreover, $V^{\perp}=\{0\}$ holds always, and ${}^{\perp} \{0\}= V$ holds if and only if $V$ is regular. Indeed, if $V$ is not regular then ${}^{\perp}\{0\}=V_{sp}-V_{sp}\neq V$.
\end{remark2}

We now give an example to illustrate some of the ideas presented above.

\begin{example2}\label{rdisjoint_esim}
Let $V=\R^3$ where $\leq$ is the standard partial ordering where $(x,y,z)\geq (0,0,0)$ if $x,y,z\geq 0$. Define $\sleq$ to be the partial ordering with the positive cone $V_{sp}=\{(x,x,z): \, x,z\geq 0\}$. Then $V$ is a quasi-regular mixed lattice space. 
Consider the following subspaces: 
$A=\{(0,0,z):z\in \R\}$, $B=\{(x,y,0):x,y\in \R \}$, $C=\{(x,x,z):x,z\in \R\}$, $D=\{(x,x,0):x\in \R\}$, $E=\{(x,x,x):x\in \R\}$.  
Then $A$ is a regular ideal and $B$ is an ideal (but not regular) such that $B=A^{\perp}$ and $A= {}^{\perp}B$. Moreover, $C=V_{sp}-V_{sp}$ is a quasi-ideal, and it is the largest quasi-ideal in $V$. Also $D=B_{sp}-B_{sp}$ is a quasi-ideal, and it is the largest quasi-ideal contained in $B$. Finally, $E$ is a regular mixed lattice subspace but not a specific ideal. 

Next, $K_1=\{(0,y,z): \, y,z\geq 0\}$ and $K_2=\{(x,0,z): \, x,z\geq 0\}$ are the maximal right disjoint cones of $D$ and $\Ss(D)=K_1\cup K_2$, which is clearly not closed under addition. Moreover, $(D^{\perp})_{p}=A_{sp}= K_1\cap K_2$ (Theorem \ref{right_disjoint_ideal2}). 
The inclusion in Theorem \ref{right_disjoint_ideal2} may be proper. To see this, let us modify the mixed lattice space, and consider $U=\R^3$ where $\leq$ is the same as above, and define $\sleq$ as the partial ordering with the positive cone $U_{sp}=\{\alpha(0,0,1)+\beta(0,1,0)+ \gamma(1,1,0): \, \alpha,\beta,\gamma\geq 0\}$. Let all the subspaces be the same as above. Then $U$ is a regular mixed lattice space, and $F=\{(0,y,0):y\in \R\}$ is a regular ideal in $U$. Now $\Ss(F)=K_2$ which is closed under addition, but $(F^{\perp})_{p}=A_{sp}\subseteq K_2$. Note also that $A$ and $B$ are still ideals in $U$, this time they are both  regular, and $B=A^{\perp}$ and $A= {}^{\perp}B$ holds. 
\end{example2}


Our definitions of the left and right disjoint complements differ from the corresponding definition in the theory of Riesz spaces. 
We recall that if $E$ is a subset of a Riesz space $L$ then the disjoint complement of $E$ is defined as $E^{\perp}=\{x\in L: |x|\wedge |y|=0 \, \textrm{ for all } y\in E\}$.

In mixed lattice spaces the generalized absolute values exist, and this naturally raises the question whether it is possible to give the definitions of the disjoint complements in terms of the absolute values, like in Riesz spaces.  
The main difficulty here is that the asymmetric generalized absolute values are not necessarily positive with respect to the specific order. To deal with this issue we introduce the notion of a symmetric absolute value, which is defined in terms of the asymmetric absolute values. It has the advantage of being 
positive with respect to both partial orderings while retaining most of the other important properties of the absolute value.

\begin{defn}\label{symmap}
Let $V$ be a mixed lattice vector space and $x\in V$. The element \, $s(x)=\frac{1}{2}(\ul{x}+\lu{x})$\, is called the \emph{symmetric generalized absolute value} of $x$. 
\end{defn}

Next we derive some basic properties of the symmetric absolute value. The first item gives useful alternative expressions and the rest of the properties show that, in many ways, the symmetric generalized absolute value behaves like the ordinary absolute value in Riesz spaces.

\begin{thm}\label{sav}
Let $V$ be a mixed lattice vector space and $x\in V$. Then the following statements hold.
\begin{enumerate}[(a)]
\item
\; $s(x)\,=\,\rupp{x}\uenv \llow{x}\, = \,\lupp{x}\,+\,\llow{x}\, = \,\rlow{x}\,+\,\rupp{x}\, = \,\rlow{x}\uenv \lupp{x}$.
\item
\; $s(\alpha x)=|\alpha|s(x)$ for all $\alpha\in\R$. 
\item
\; $s(x)\sgeq 0$ \; and \; $s(x)\geq 0$. Moreover, \, $s(x)= 0$ \, if and only if \, $x=0$.
\item
\; $x\sgeq 0$ \, if and only if \, $x=s(x)$. In particular,  $s(s(x))=s(x)$.
\item
\;
$s(x+y)\leq s(x)+s(y)$. 
\item
\; $x\lenv y + y\lenv x = x+y-s(x-y)$.
\end{enumerate}
\end{thm}

\begin{bproof}
\begin{enumerate}[(a)]
\item 
\; All the equalities apart from the first one were given in Theorem \ref{absval}(h). The first equality follows from
$$
2s(x)=\ul{x}+\lu{x} = \lupp{x}\,+\,\llow{x}\, + \,\rlow{x}\,+\,\rupp{x}\, = 2(\rlow{x}\,+\,\rupp{x}) = 2(\rupp{x}\uenv \llow{x}),
$$
where we used Theorem \ref{absval} (c) and (h). 
\item
\; This follows from Theorem \ref{absval} (l) and (m).
\item 
\; Since $x^{u}\sgeq 0$ and $x^{l}\sgeq 0$ we have $s(x)=x^{u}+x^{l}\sgeq 0$. Also, $\lupp{x}\geq 0$ and $\llow{x}\geq 0$ imply $s(x)=\lupp{x}+\llow{x}\geq 0$. It is clear that $x=0$ implies $s(x)=0$. Assume then that $s(x)=0$. Now $s(x)=\lupp{x}+\llow{x}=0$, \,or \,$\lupp{x}=-\llow{x}$. Hence 
$0\leq\lupp{x}=-\llow{x}\leq 0$, which implies $\lupp{x}=0$. On the other hand, $s(x)=\rupp{x}+\rlow{x}=0$, \,or \, $\rlow{x}=-\rupp{x}$. So, $0\sleq\rlow{x}=-\rupp{x}\sleq 0$, which in turn implies $\rlow{x}=0$. Consequently, $x=\lupp{x}-\rlow{x}=0$.
\item 
\; This follows immediately from (a), (c) and Theorem \ref{absval} (i).
\item
\; Using (a) together with Theorem \ref{absval} (e) and (f) we get
$$
s(x+y)=\lupp{(x+y)}+\llow{(x+y)}\leq \lupp{x}+\lupp{y} +\llow{x}+\llow{y} =s(x)+s(y).
$$
\item
\; This follows by adding the two identities given in Theorem \ref{absval} (n).
\end{enumerate}
\end{bproof}

We can now characterize ideals in terms of the symmetric absolute value.

\begin{thm}\label{ideal_charact2}
Let $V$ be a mixed lattice space and $A$ a subspace of $V$. The following conditions are equivalent. 
\begin{enumerate}[(a)]
\item
$A$ is an ideal.
\item
If\, $s(x)\leq s(y)$ \,and \,$y\in A$ \,then \,$x\in A$.
\end{enumerate}
\end{thm}

\begin{bproof}
Assume that $(a)$ holds and let $s(x) \leq s(y)$ with $y\in A$. 
Then $s(y)\in A$, and the inequalities  $0\leq \rlow{x}\leq s(x)\leq s(y)$ and $0\leq \lupp{x}\leq s(x)\leq s(y)$ imply that $\lupp{x},\rlow{x}\in A$ and so $\lupp{x}-\rlow{x}=x\in A$.

Conversely, assume that $(b)$ holds and let $y\in A$. Then by Theorem \ref{sav}  
$0\sleq \rupp{y}= s(\rupp{y})\leq s(y)$.
It follows that $\rupp{y}\in A$ and so $A$ is a mixed lattice subspace. 
Next, assume that $0\leq x\leq y$ with $y\in A$. Then  
$\rlow{y}=\rlow{x}=0$ and so $s(y)= \rupp{y}$ and $s(x)= \rupp{x}$. Now the above assumption implies that $\rupp{x}=s(\rupp{x})\leq \rupp{y}=s(\rupp{y})\leq s(y)$. It follows that $\rupp{x}=s(x)\in A$ and since $s(x)\leq s(s(x))$ by Theorem \ref{sav}(d), we infer that $x\in A$ and so $A$ is $(\leq)$-order convex, and hence an ideal.  
\end{bproof}

Next we will show that the right disjoint complement can also be given in terms of the symmetric absolute value, and for Riesz subspaces, the usual Riesz space definition of the disjoint complement can thus be viewed as a special case of the next result.

\begin{thm}\label{rdc_in_terms_of_sav}
If $A$ is a mixed lattice subspace then the right disjoint complement is given by 
$A^{\perp}=\{x\in V: s(z)\lenv s(x)=0 \; \textrm{ for all }\, z\in A\}$.
\end{thm}

\begin{bproof}
Let $X=\{x\in V: s(z)\lenv s(x)=0 \; \textrm{ for all }\, z\in A\}$. We first note that $A_{sp}=\{s(z):z\in A\}$ and so $X=\{x\in V: z\lenv s(x)=0 \; \textrm{ for all }\, z\in A_{sp}\}$. Let $x,y\in X$ and $z\in A_{sp}$. Since $s(x)\sgeq 0$ and $s(y)\sgeq 0$ we have 
$$
0\leq z\lenv s(x+y) \leq z\lenv (s(x)+s(y))\leq z\lenv s(x)+z\lenv s(y)=0,
$$
and so $z\lenv s(x+y)=0$ (here we used Theorem \ref{sav} and Lemma \ref{disjoint_sum_lemma}). This shows that $X$ is a subspace. To show that $X$ is an ideal, we note that if $s(x)\leq s(y)$ and $y\in X$ then $0\leq z\lenv s(x) \leq z\lenv s(y)=0$, and so $x\in X$ and by Theorem \ref{ideal_charact2} $X$ is an ideal.

Clearly, $A^{\perp}\subseteq X$. If $x\in X$ then $s(x)\in A^{\perp}_{sp}$, and it follows from  
$0\leq \rupp{x}\leq s(x)$ and $0\leq \llow{x}\leq s(x)$ that $x=\rupp{x}-\llow{x}\in A^{\perp}$. Hence $A^{\perp}=X$.
\end{bproof}

Theorem \ref{rdc_in_terms_of_sav} thus gives an alternative (but equivalent) definition of the right disjoint complement by means of the symmetric absolute value. The reason this works is implied by Theorem \ref{ideal_charact2}, for if $A$ is an ideal then $x\in A$ if and only if $s(x)\in A$. However, if $A$ is a specific ideal, then $x\in A$ implies $s(x)\in A$, but not conversely in general. This is the fundamental reason why we cannot similarly characterize the left disjoint complement $^{\perp}A$ in terms of the symmetric absolute value.

In the next section we consider the situation where $A$ is a band such that $A= ({}^{\perp}A)^{\perp}$. We will now give sufficient conditions for this to hold. 
The next couple of results have their counterparts in the theory of Riesz spaces, and the methods used in their proofs are also similar. First we recall some terminology. A mixed lattice space $V$ is called $(\leq)$-\emph{Archimedean} if the condition $nx\leq y$ for all $n\in \N$ implies $x\leq 0$. Similarly, $V$ is called $(\sleq)$-\emph{Archimedean} if the condition $nx\sleq y$ for all $n\in \N$ implies $x\sleq 0$. It is easy to see that if $V$ is $(\leq)$-Archimedean then it is also $(\sleq)$-Archimedean, but the converse is not true as the next example shows.

\begin{example2}
Let $V=(\R^2,\leq,\sleq)$, where $\leq$ is the lexicographic ordering defined as
$$
(x_1,x_2)\leq (y_1,y_2) \quad \iff \quad  ( \; x_1< y_1 \quad \textrm{or} \quad x_1=y_1 \quad \textrm{and} \quad x_2\leq y_2 \;),
$$
and $\sleq$ is the usual coordinatewise ordering $(x_1,x_2)\sleq (y_1,y_2)$ if $x_1\leq y_1$ and $x_2\leq y_2$. Then $V$ is a mixed lattice space which is $(\sleq)$-Archimedean but not $(\leq)$-Archimedean.
\end{example2}

Next we will show that in $(\sleq)$-Archimedean mixed lattice spaces every quasi-band $A$ has the property that $A_{sp}= (({}^{\perp}A)^{\perp})_{sp}$. 
First we consider the following order-denseness property of quasi-ideals with respect to their second right disjoint complements.

\begin{prop}\label{orderdense_ideal}
If $A$ is a quasi-ideal and $x\in ({}^{\perp}A)^{\perp}$ is a non-zero element such that $x\sgeq 0$, then there exists a non-zero element $y\in A$ such that $0\sleq y\sleq x$.
\end{prop}

\begin{bproof}
Let $0\sleq x\in ({}^{\perp}A)^{\perp}$ be a non-zero element. Then $x\notin {}^{\perp}A$, and there exists an element $0\sleq w\in A$ such that $x\lenv w \neq 0$. Since $A$ is a quasi-ideal, the inequality $0\sleq x\lenv w \leq w$ implies that $x\lenv w \in A$. The element $y=x\lenv w$ then satisfies $0\sleq y\sleq x$, as required. 
\end{bproof}

\begin{remark2}
The above proposition holds also if $A=\{0\}$, because then ${}^{\perp}A=V_{sp}-V_{sp}$ and $({}^{\perp}A)^{\perp}=\{0\}$,  so there are no non-zero elements in $({}^{\perp}A)^{\perp}$ which makes the proposition vacuously true.
\end{remark2}

In $(\sleq)$-Archimedean mixed lattice spaces the order-denseness property of the preceding proposition has the following consequence.

\blue{
\begin{thm}\label{orderdense_archim2}
Let $V$ be a 
$(\sleq)$-Archimedean mixed lattice space and $A$ a quasi-ideal in $V$. If $M_u=\{v\in A: 0\sleq v\sleq u\}$ then $u=\spsup M_u$ for every $u\in (({}^{\perp}A)^{\perp})_{sp}$.   
\end{thm}

\begin{bproof}
Let $0\sleq u\in ({}^{\perp}A)^{\perp}$ and $M_u=\{v\in A: 0\sleq v\sleq u\}$.  
Clearly, $u$ is a specific upper bound of $M_u$. Assume that $u\neq \spsup M_u$. Then there exists some  
specific upper bound $w_0$ of $M_u$ such that $u\sleq w_0$ does not hold. 
If we put $w=u\lenv w_0$ then $0\sleq w\sleq u$ holds, and $w\in ({}^{\perp}A)^{\perp}$, since $({}^{\perp}A)^{\perp}$ is an ideal. Moreover, $0\sleq u-w \in ({}^{\perp}A)^{\perp}$ and so we can apply Proposition \ref{orderdense_ideal} to find a non-zero element $z\in A$ such that $0\sleq z\sleq u-w\sleq u$. 
In particular, 
we have $z\sleq u$ and $z\sleq w_0$, and consequently, $z\sleq w$. Hence 
$$
0\sleq 2z=z+z\sleq (u-w)+w=u,
$$ 
and clearly $2z\in A$, so in fact we have $2z\in M_u$. Hence,  
we can repeat the same argument for the element $2z$. That is, 
$$
0\sleq 3z=z+2z\sleq (u-w)+w=u,
$$ 
and so $3z\in M_u$. Continuing this way we find that $nz\in M_u$ for all $n$, and so   
$0\sleq nz\sleq u$ for all $n$. But $0\sleq nz\neq 0$, which is impossible since $V$ is $(\sleq)$-Archimedean.  
Hence, $u=\spsup M_u$, and the proof is complete.
\end{bproof}

The following result is now a rather immediate consequence of the 
preceding theorem.

\begin{cor}\label{archim_2_complement}
If $V$ is a 
$(\sleq)$-Archimedean mixed lattice space and $A$ is a quasi-band  
in $V$, then $A_{sp}= (({}^{\perp}A)^{\perp})_{sp}$. 
\end{cor}
}

\section{Direct sum decompositions and projections}
\label{sec:s6}

The sum of subspaces $A$ and $B$ of a mixed lattice space $V$ is called a \emph{direct sum} if $A\cap B=\{0\}$, and this is denoted by $A\oplus B$.
If $V$ has a direct sum decomposition $V=A\oplus B$ then each element $x\in V$ has a unique representation $x=x_1+x_2$, where $x_1\in A$ and $x_2\in B$.  
The elements $x_1$ and $x_2$ are called \emph{components} of $x$ in $A$ and $B$, respectively.
A vector space can usually be written as a direct sum of subspaces in many ways, but we are mainly concerned with direct sum decompositions that are compatible with the mixed lattice structure in the sense that the order structure of $V$ 
is determined by its direct sum components.

In this section we give some results concerning direct sum decompositions of a 
mixed lattice space in terms of ideals and specific ideals. First we show that if a mixed lattice space $V$ can be written as a direct sum of a regular specific ideal $A$ and an ideal $B$ then $A$ and $B$ are necessarily the disjoint complements of each other. Moreover, $A$  
has stronger band properties than those implied by Theorem \ref{left_disjoint_comp_is_band}.

\begin{thm}\label{direct_sum}
If $A$ is a regular specific ideal and $B$ is an ideal  
such that $V=A\oplus B$ then $A$ is a regular specific band and $B$ is a \blue{weak band}  
such that $A= {}^{\perp}B = {}^{\perp}(A^{\perp})$ and $B= {A}^{\perp} =({}^{\perp}B)^{\perp}$. Moreover, $V_{sp}=A_{sp}+B_{sp}$.
\end{thm}

\begin{bproof}
Assume that $V=A\oplus B$, that is, $A\cap B=\{0\}$. Let $u\in A$ and $v\in B$ such that $u\sgeq 0$ and $v\geq 0$. Then, since $B$ is an ideal, 
it follows from  $0\leq u\lenv v \leq v$ that $u\lenv v\in B$, and consequently $\rupp{(u\lenv v)}\in B$.  
On the other hand,  
$u\lenv v\sleq u$ and by Theorem \ref{new_ineq} we have
$$
0\sleq \rupp{(u\lenv v)} = 0\uenv (u\lenv v) \sleq 0 \uenv u =\rupp{u}=u,
$$
and so $\rupp{(u\lenv v)}\in A$, since $A$ is a specific ideal. Hence, we have $\rupp{(u\lenv v)}\in A\cap B$, and so $\rupp{(u\lenv v)}=0$. Since $u\lenv v\geq 0$, this implies that $u\lenv v =0$. Hence, $A\subseteq {}^{\perp}B$ and $B\subseteq {A}^{\perp}$.

To prove the reverse inclusions, let $x\in {}^{\perp}B$. Since $V=A\oplus B$, the element $x$ has a unique decomposition $x=x_1+x_2$ with $x_1\in A$ and $x_2\in B$. But we showed above that $A\subseteq {}^{\perp}B$, and so $x_1\in {}^{\perp}B$. 
Since ${}^{\perp}B$ is a subspace, we have $x-x_1=x_2\in{}^{\perp}B$. Thus  
$x_2=0$ (by Theorem \ref{properties_of_complements}) and $x=x_1\in A$. Hence, ${}^{\perp}B\subseteq A$ and so $A= {}^{\perp}B$. 
From this it follows that $A^{\perp}= ({}^{\perp}B)^{\perp}$ and  
$V=A\oplus B={}^{\perp}B \oplus B$. Thus, every $x\in A^{\perp}=({}^{\perp}B)^{\perp}$ has a unique decomposition $x=x_1+x_2$ with $x_1\in {}^{\perp}B$ and $x_2\in B\subseteq ({}^{\perp}B)^{\perp}$, and so by Theorem \ref{properties_of_complements} we have $x_1=x-x_2 \in ({}^{\perp}B)^{\perp}$. But then $x_1\in {}^{\perp}B\cap ({}^{\perp}B)^{\perp}= \{0\}$,  
so we deduce that $x=x_2\in B$, and therefore $A^{\perp} \subseteq B$, and so $A^{\perp}=B$. 
Hence, we have $A= {}^{\perp}B = {}^{\perp}(A^{\perp})$ and $B= A^{\perp}=({}^{\perp}B)^{\perp}$.

\blue{
The equality $V_{sp}= A_{sp}+B_{sp}$ holds by  
Theorem \ref{specific_direct_sum1}(b). 
By Theorem \ref{right_disjoint_ideal}, $B$ is a weak band,  
and by Theorem \ref{left_disjoint_comp_is_band}, $A$ is a weak specific band, but we must show that $A$ is a specific band. 
For this, let $E$ be a subset of $A$ such that $w=\spsup E$ exists in $V$. Then for any $x\in E$ we have $z=w-x\sgeq 0$, and so $z=z_1+z_2$ with $z_1\in A_{sp}$ and $z_2\in B_{sp}$. On the other hand, we have $w=w_1+w_2$ with $w_1\in A$ and $w_2\in B$, and therefore $z=(w_1-x)+w_2$ where $w_1-x\in A$. Since the representation of $z$ is unique, we have $z_1=w_1 -x \in A_{sp}$ and $z_2=w_2 \in B_{sp}$. Hence,  
$x\sleq w_1 \sleq w_1+w_2 = w$. Since this holds for any $x\in E$,  
it follows that $w_2 =0$ and $w=w_1\in A$, and so $A$ is a specific band.} 
\end{bproof}

For a regular mixed lattice space we have the following result.

\begin{thm}\label{specific_direct_sum2}
Let $A$ be a regular specific 
ideal and $B$ an ideal 
such that $V=A\oplus B$.  
Then $V$ is regular if and only if $A$ and $B$ are both regular. 
\end{thm}

\begin{bproof}
If $A$ and $B$ are both regular then clearly $V$ is also regular. If $V$ is regular and $x\in B$ then $x\sleq u$ for some $u\in V_{sp}$. Now 
by the preceding theorem $u=u_1+u_2$ with $u_1\in A_{sp}$ and $u_2\in B_{sp}$. The inequality $x\sleq u_1+u_2$ is equivalent to $0\sleq u_1+(u_2-x)$, and, again by Theorem \ref{direct_sum}, we have $u_2-x\in B_{sp}$.  
This implies that $B$ is regular. 
\end{bproof}


If there is a weak band $A$ such that $V=\ldc{A}\oplus A$ then the components of specifically positive elements in $A$ and $\ldc{A}$ are given by the next theorem. 

\begin{thm}\label{sp_components}
Let $A$ be a weak band such that $V=\ldc{A} \oplus A$. Then for every $x\in V_{sp}$ the elements
$$
x_1=\spsup\{v\in A: 0\sleq v\sleq x\} \quad \textrm{and} \quad x_2=\spsup\{v\in \ldc{A}: 0\sleq v\sleq x\}
$$
exist, and $x_1$ and $x_2$ are the components of $x$ in $A$ and $\ldc{A}$, respectively.
\end{thm}

\begin{bproof}
Let $x\in V_{sp}$ have the decomposition $x=x_1+x_2$ with $x_1\in A$ and $x_2\in \ldc{A}$. Let $M=\{v\in A: 0\sleq v\sleq x\}$. If $v\in M$ then the element $x-v\sgeq 0$ has the unique decomposition $x-v=(x_1-v)+x_2$ where $x_1-v\in A_{sp}$ and $x_2\in (\ldc{A})_{sp}$, by Theorem \ref{direct_sum}. Thus, $v\sleq x_1$ and so $x_1$ is a $(\sleq)$-upper bound of $M$. But we also have $0\sleq x_1=x-x_2\sleq x$, and so $x_1\in M$. This shows that $x_1=\spsup M$. The formula for $x_2$ is proved similarly.
\end{bproof}

In accordance with the Riesz space terminology, the components given by the above theorem are called \emph{specific projections} of $x$ on $A$ and  $\ldc{A}$. We will return to the discussion of projection elements later.

In the above theorems we considered a direct sum $V=A\oplus B$ such that $V_{sp}=A_{sp}+B_{sp}$. 
If also $V_{p}=A_{p}+B_{p}$ holds then we obtain stronger results, and this motivates the following definition.

\begin{defn}
Let $A$ and $B$ be subspaces of a mixed lattice space $V$ such that $V=A\oplus B$. This 
direct sum is called a \emph{mixed-order direct sum} if $V_{sp}=A_{sp}+B_{sp}$ and $V_p=A_p+B_p$. 
\end{defn}

Mixed-order direct sums can be characterized as follows.

\begin{thm}\label{positive_decomp2}
Let $A$ and $B$ be subspaces of $V$ such that $V=A\oplus B$. Then the following are equivalent. 
\begin{enumerate}[(a)]
\item
$V=A\oplus B$ is a mixed-order direct sum 
\item
$A$ is a regular band and $B$ is a band  such that $A= {}^{\perp}B$ and $B= {A}^{\perp}$. 
\end{enumerate}
\end{thm}

\begin{bproof}
Assuming that $V=A\oplus B$ is a mixed-order direct sum, we first show that $A$ is $(\leq)$-order convex. Let $0\leq v \leq u$ with $u\in A$. If we put $w=u-v$ then $u=v+w$ and $v,w\geq 0$. Now by assumption $v$ and $w$ have representations $v=v_1+v_2$ and $w=w_1+w_2$ where $v_1,w_1\in A_{p}$ and $v_2,w_2\in B_p$. Hence, $u=(v_1+v_2)+(w_1+w_2)$, or $v_2+w_2=u-(v_1+w_1)$. $A$ and $B$ are subspaces, so $v_2+w_2\in B$ and $u-(v_1+w_1)\in A$. But then $v_2+w_2\in A\cap B =\{0\}$, and it follows that $v_2=w_2=0$. Thus, $v=v_1\in A_{p}$ and this shows that $A$ is $(\leq)$-order convex. 

Next we need to prove that $A$ is a mixed lattice subspace.  
By assumption, the elements $\rupp{x}$ and $\llow{x}$ can be written as $\rupp{x}=a_1+b_1$ and $\llow{x}=a_2+b_2$ where $a_1\in A_{sp}$, $b_1\in B_{sp}$ and  $a_2\in A_{p}$,  $b_2\in B_p$.
Now $x=\rupp{x}-\llow{x}=(a_1-a_2)-(b_2-b_1)$, so $b_2-b_1=x-(a_1-a_2)$. Since $A$ and $B$ are subspaces, we have $b_2-b_1\in B$ and $x-(a_1-a_2)\in A$. But then $b_2-b_1\in A\cap B =\{0\}$, and it follows that  $x=a_1-a_2$, 
and so we have two representations  $x=\rupp{x}-\llow{x}$ and $x=a_1-a_2$ with $a_1\in A_{sp}$ and $a_2\in A_p$. By Theorem \ref{disjoint} we have $0\leq  \rupp{x} \leq a_1$ and $0\leq  \llow{x} \leq a_2$. It follows that $\llow{x}, \rupp{x} \in A$, since $A$ is order-convex. This shows that $A$ is a mixed lattice subspace, and hence an ideal. 
Similar arguments show that $B$ is also an ideal.

To show that $A$ is regular, we note that $0\leq x\lenv y\leq x$ and $0\leq x\lenv y\leq y$ for all $x\in A_{sp}$ and $y\in B_p$, and since $A$ and $B$ are ideals such that $A\cap B =\{0\}$, it follows that $x\lenv y=0$ for all $x\in A_{sp}$ and $y\in B_p$. Hence, $A\subseteq  {}^{\perp}B$. Next, let $0\leq z\in {}^{\perp}B$. By assumption we can write $z=a+b$ where $a\in A_p$ and $b\in B_p$. Then $b=z-a\in {}^{\perp}B$, because $A\subseteq {}^{\perp}B$ and ${}^{\perp}B$ is a subspace. Thus, $b\in B\cap {}^{\perp}B=\{0\}$, by Theorem  \ref{properties_of_complements}. Hence, $z=a\in A$ and so $A= {}^{\perp}B$, which is a regular weak band  
by Theorem \ref{left_disjoint_comp_is_band}. It now follows from Theorem \ref{direct_sum} that $B$ is a weak band such that $B= {A}^{\perp}$. 
\blue{To see that $B$ is actually a band, let $E$ be a subset of $B$ such that $w=\sup E$ exists in $V$. Then for any $x\in E$ we have $z=w-x\geq 0$, and so $z=z_1+z_2$ with $z_1\in A_{p}$ and $z_2\in B_{p}$. On the other hand, $w=w_1+w_2$ with $w_1\in A$ and $w_2\in B$, and so $z=w_1 + (w_2-x)$ where $w_2-x\in B$. Since this representation is unique, we have $z_1=w_1 \in A_{p}$ and $z_2=w_2-x \in B_{p}$. Hence,  
$x\leq w_2 \leq w_1+w_2 = w$. Since this holds for any $x\in E$,  
it follows that $w_1 =0$ and $w=w_2\in B$, and so $B$ is a band. A similar argument shows that $A$ is also a band.}

To prove 
the implication $(b)\Rightarrow (a)$, assume that $A$ is a regular band such that $V=A\oplus A^{\perp}$. If $W=\aspp-\aspp$ 
then by Theorem \ref{smallest_ideal_inside_ideal} $W$ is a quasi-ideal, so by Theorem \ref{specific_direct_sum1} we have  
$V_{sp}=A_{sp}+(A^{\perp})_{sp}=A_{sp}+W_{sp}$. Then $V_{p}=A_{p}+(A^{\perp})_{p}$, by Theorem \ref{sum_of_ideals2}, and hence $V=A\oplus A^{\perp}$ is a mixed-order direct sum.
\end{bproof}

Now, if $V=A\oplus A^{\perp}$ is a mixed-order direct sum then the left and right disjoint complements of $A$ are very closely related, and they are equal if $V$ is regular.

\begin{thm}\label{mods_relations}
If $V=A\oplus A^{\perp}$ is a mixed-order direct sum then $A= {}^{\perp}(A^{\perp}) = (A^{\perp})^{\perp}$. Moreover, ${}^{\perp}A \subseteq A^{\perp}$, and if $A^{\perp}$ is regular then ${}^{\perp}A = A^{\perp}$. In this case, 
$A= {}^{\perp}(A^{\perp}) = (A^{\perp})^{\perp} = {}^{\perp}({}^{\perp}A) = ({}^{\perp}A)^{\perp}$. In particular, all these equalities hold if $V$ is regular.
\end{thm}

\begin{bproof}
The equality $A= {}^{\perp}(A^{\perp})$ follows from Theorem \ref{positive_decomp2}, and since $A$ and $A^{\perp}$ are both ideals, it then follows from Theorem \ref{properties_of_complements}(e) that ${}^{\perp}(A^{\perp})=A=(A^{\perp})^{\perp}$. 
By Theorem \ref{properties_of_complements}(d) $(A^{\perp})_{sp}\subseteq ({}^{\perp}A)_{sp}$. To prove the reverse inclusion, let $x\in ({}^{\perp}A)_{sp}$. Then by assumption we can write $x=x_1+x_2$ where $x_1\in A_{sp}$ and $x_2\in (A^{\perp})_{sp}\subseteq ({}^{\perp}A)_{sp}$. Then for every $z\in A_{sp}$ we have $x_2\lenv z=0$, and it follows by by Lemma \ref{disjoint_sum_left} that $0= x\lenv z = (x_1+x_2)\lenv z = x_1\lenv z$. Hence $x_1\in ({}^{\perp}A)_{sp}$, and so $x_1\in {}^{\perp}A\cap A=\{0\}$. Thus $x=x_2\in (A^{\perp})_{sp}$, and so $({}^{\perp}A)_{sp}=(A^{\perp})_{sp}$. Since ${}^{\perp}A=({}^{\perp}A)_{sp}-({}^{\perp}A)_{sp}$, it follows that ${}^{\perp}A\subseteq A^{\perp}$, and if $A^{\perp}$ is regular then the equality ${}^{\perp}A = A^{\perp}$ holds. In this case we therefore have $A= {}^{\perp}(A^{\perp}) = (A^{\perp})^{\perp} = {}^{\perp}({}^{\perp}A) = ({}^{\perp}A)^{\perp}$. In particular, this holds if $V$ is regular, by Theorem \ref{specific_direct_sum2}.
\end{bproof}

The next theorem shows that disjoint components of elements in a mixed lattice space have similar properties as in Riesz spaces. In particular, the symmetric absolute value behaves as one would expect.

\begin{thm}\label{disjoint_sav}
Let $V=A\oplus B$ be a mixed-order direct sum. If $x\in A$ and $y\in B$ then the following hold.
\begin{enumerate}[(a)]
\item
$\rupp{(x+y)}=\rupp{x}+\rupp{y}$ \; and \; $\llow{(x+y)}=\llow{x}+\llow{y}$.
\item
$\lupp{(x+y)}=\lupp{x}+\lupp{y}$ \; and \; $\rlow{(x+y)}=\rlow{x}+\rlow{y}$.
\item
$\ul{(x+y)}=\ul{x}+\ul{y}$ \; and \; $\lu{(x+y)}=\lu{x}+\lu{y}$.
\item
$s(s(x)-s(y))=s(x+y)=s(x-y)=s(x)+s(y)=s(x)\uenv s(y)=s(y)\uenv s(x)=\spsup\{s(x),s(y)\}$.
\end{enumerate}
\end{thm} 

\begin{bproof}
To prove (a), we start by noting that $0\leq\rupp{x}\lenv (\llow{x}+\llow{y})\leq \rupp{x}\lenv (\llow{x}+\rupp{(\llow{y})})$. Now $\rupp{x}\lenv \llow{x}=0$ (by Theorem \ref{absval}) and  $\rupp{x}\lenv \rupp{(\llow{y})}=0$ (since $\rupp{(\llow{y})}\in B_{sp}$). Hence, by Lemma \ref{disjoint_sum_lemma} we have $\rupp{x}\lenv (\llow{x}+\rupp{(\llow{y})})=0$, and consequently, $\rupp{x}\lenv (\llow{x}+\llow{y})=0$. Now let $w=(\rupp{x}+\rupp{y})\lenv (\llow{x}+\llow{y})$. Then $w\leq \llow{x}+\llow{y}$ and, since $\rupp{y}\sgeq 0$, by Theorem \ref{new_ineq} we also have 
$$
w\sleq (\rupp{x}+\rupp{y})\lenv (\rupp{y}+\llow{x}+\llow{y})=\rupp{y} +  \rupp{x}\lenv (\llow{x}+\llow{y}) = \rupp{y}.
$$
This implies that $0\leq w\leq \rupp{y}\lenv (\llow{x}+\llow{y})$.
Next we note that $\rupp{y}\in B_{sp}=(A^{\perp})_{sp}\subseteq ({}^\perp A)_{sp}$ (by Theorem \ref{properties_of_complements}(d)), and $\rupp{(\llow{x})}\in A_{sp}$, so  $0\leq \rupp{y}\lenv \llow{x} \leq\rupp{y}\lenv \rupp{(\llow{x})}=0$. Moreover, $\rupp{y}\lenv \llow{y}=0$, so we can use Lemma \ref{disjoint_sum_lemma} again to obtain $0\leq \rupp{y}\lenv (\llow{x}+\llow{y})\leq \rupp{y}\lenv (\rupp{(\llow{x})}+\llow{y})=0$. 
Hence, we have shown that $\rupp{y}\lenv (\llow{x}+\llow{y})=0$, and consequently, $w=0$. 

Since $x+y=(\rupp{x}+\rupp{y})-(\llow{x}+\llow{y})$ and $(\rupp{x}+\rupp{y})\lenv (\llow{x}+\llow{y})=0$, we have $\rupp{(x+y)}=\rupp{x}+\rupp{y}$ and $\llow{(x+y)}=\llow{x}+\llow{y}$, by Theorem \ref{disjoint}. 
Similar reasoning proves (b), and adding the equalities in (a) and (b) gives the equalities in (c). Adding the equalities in (c) then gives $s(x+y)=s(x)+s(y)$. Since this holds for all $x$ and $y$, we can replace $y$ by $-y$ to get $s(x-y)=s(x)+s(-y)=s(x)+s(y)$. The equality 
$s(s(x)-s(y))=s(x)+s(y)$ then follows immediately from Theorem \ref{sav}(f) by replacing $x$ with $s(x)$ and $y$ with $s(y)$, and using the fact that $s(x)\lenv s(y)=s(y)\lenv s(x)=0$. Indeed, since  
$s(x)\in A_{sp}$ and $s(y)\in B_{sp}=({A}^{\perp})_{sp}$, then by Theorem \ref{properties_of_complements}(d) we have $s(y)\in ({}^{\perp}A)_{sp}$, and so $s(x)\lenv s(y)=s(y)\lenv s(x)=0$. Consequently, 
$s(x)\uenv s(y)=s(y)\uenv s(x)=s(x)+s(y)=\spsup\{s(x),s(y)\}$. Here the last equality follows by noting that $s(x)+s(y)$ is clearly a $(\sleq)$-upper bound of $\{s(x),s(y)\}$, and if $s(x)\sleq c$ and $s(y)\sleq c$ then $s(x)\uenv s(y)\sleq c$, completing the proof.
\end{bproof}

If $C$ is a mixed lattice cone then the following inequalities hold for all $x,y,z\in C$. 
(For the proofs we refer to \cite[pp.13]{ars} and \cite[Theorem 3.6.]{ars})

\begin{equation}\label{kakka1}
z\lenv (x+y) \leq z\lenv x + z\lenv y
\end{equation}
\begin{equation}\label{kakka2}
(x+y)\lenv z \sleq x\lenv z + y\lenv z
\end{equation}

\vspace{0.3cm}
With these inequalities we obtain the following formulae for the mixed envelopes of specifically positive elements.

\begin{prop}\label{disjoint_sp_element_prop}
Let $V=A\oplus B$ be a mixed-order direct sum. 
If $x,y\in V_{sp}$ are elements with components $x=x_1+x_2$ and $y=y_1+y_2$ where $x_1,y_1\in A_{sp}$ and $x_2,y_2\in B_{sp}$, then 
$$
x\lenv y=x_1\lenv y_1 + x_2\lenv y_2 \quad \textrm{ and } \quad x\uenv y=x_1\uenv y_1 + x_2\uenv y_2.
$$
\end{prop}

\begin{bproof}
$A_{sp}$ and $B_{sp}$ are mixed lattice cones, so  
using \ref{kakka1} and \ref{kakka2} together with the fact that $x_1\lenv y_2 = 0 = x_2\lenv y_1$ \blue{(where the last equality follows from Theorem \ref{mods_relations}, since $y_1\in A={}^{\perp}B \subseteq B^{\perp}$)} we get
$$
\begin{array}{lcl}
x\lenv y & = & (x_1+x_2)\lenv(y_1+y_2) \\
& \sleq  &  x_1 \lenv(y_1+y_2) + x_2 \lenv(y_1+y_2) \\
& \leq & x_1\lenv y_1 + x_1\lenv y_2 + x_2\lenv y_1 + x_2\lenv y_2 \\
& = & x_1\lenv y_1 + x_2\lenv y_2.
\end{array}
$$
On the other hand, we have $x_1\lenv y_1 + x_2\lenv y_2\sleq x_1+x_2=x$ and $x_1\lenv y_1 + x_2\lenv y_2\leq y_1+y_2=y$, so $x_1\lenv y_1 + x_2\lenv y_2\leq x\lenv y$, and the first identity follows. Similar reasoning shows that $y\lenv x=y_1\lenv x_1 + y_2\lenv x_2$.

Substituting these into $x\uenv y=x+y-y\lenv x$ gives 
$$
\begin{array}{lcl}
x\uenv y & = & x_1+x_2 +y_1+y_2 - (y_1\lenv x_1 + y_2\lenv x_2) \\
& = & (x_1+y_1 -y_1\lenv x_1) + (x_2+y_2-y_2\lenv x_2)\\
& = & x_1 \uenv y_1 + x_2 \uenv y_2.
\end{array}
$$
\end{bproof}


If $V=A\oplus A^{\perp}$ is a mixed order direct sum then the components of a positive element in $A$ and $A^{\perp}$ are given by the next theorem, which is proved exactly as Theorem \ref{sp_components}. 

\begin{thm}\label{components}
Let $A$ be a band such that $V=A\oplus A^{\perp}$ is a mixed order direct sum. Then for every $x\in V_{p}$ the elements
$$
x_1=\sup\{v\in A: 0\leq v\leq x\} \quad \textrm{and} \quad x_2=\sup\{v\in \rdc{A}: 0\leq v\leq x\}
$$
exist, and $x_1$ and $x_2$ are the components of $x$ in $A$ and $\rdc{A}$, respectively.
\end{thm}

%

As in the theory of Riesz spaces, it turns out that the existence of a mixed order direct sum is equivalent to the existence of the associated order projection. Let $V=A\oplus A^{\perp}$ be a mixed order direct sum. If $x\in V$ has the components $x_1\in A$ and $x_2\in A^{\perp}$ then we define the mapping $P_A:V\to V$ by $P_A (x)=x_1$. We can immediately see that $P_A$ has the following properties:

$$ 
\begin{array}{ll}
(P1)   & \; P_A \, \textrm{ is a linear operator and } \, P_A^2=P_A.  \\ 
(P2)  & \; 0\leq P_A (x) \leq x \; \textrm{ for every } \, x\geq 0. \\
(P3)  & \; 0\sleq P_A (x) \sleq x \; \textrm{ for every } \, x\sgeq 0. \\[1ex]
\end{array} 
$$

A mapping with these properties is called a \emph{mixed order projection} and the associated band $A$ is called a \emph{projection band}. (Note that the specific projection elements in Theorem \ref{sp_components} can be given similarly in terms of a specific projection operator that has the properties (P1) and (P3)).

Conversely, if $P$ is a mixed order projection on $V$ then there exists a corresponding mixed order direct sum decomposition of $V$. Apart from a few minor modifications, the proof is essentially the same as in the case of Riesz spaces (see \cite[Theorem 11.4]{zaa}). 

\begin{thm}\label{projections}
If $P:V\to V$ is a mapping with properties (P1)--(P3) then there exists a regular band $A$ such that $V=A\oplus A^{\perp}$ is a mixed order direct sum and $P$ is the associated projection on $A$ and $I-P$ is the projection on $A^{\perp}$.
\end{thm}

\begin{bproof}
Let $A=\{Px:x\in V\}$ and $B=\{(I-P)x:x\in V\}$. Then $A$ and $B$ are clearly subspaces of $V$. If $x\in A\cap B$ then $x=Py$ and $x=(I-P)z$ for some $y,z\in V$. Using the property (P1) we get
$$
x=Py=P(Py)=P(I-P)z=Pz-P^2z=0.
$$ 
This shows that $A\cap B=\{0\}$, so $A\oplus B$ is a direct sum. Moreover, for any $x\in V$ we have $x=Px+(x-Px)=Px+(I-P)x$, so $V=A\oplus B$. It now follows from (P2) that if $x\geq 0$ then $Px\geq 0$ and $0\leq x-Px=(I-P)x$. Hence, $V_p=A_p+B_p$. Similarly, (P3) implies that $V_{sp}=A_{sp}+B_{sp}$, and so $V=A\oplus B$ is a mixed order direct sum. It then follows from Theorem \ref{positive_decomp2} that $A$ is a regular band and $B=A^{\perp}$. Clearly, $P$ and $I-P$ are the projections on $A$ and $A^{\perp}$, respectively.
\end{bproof}

If $x\geq 0$ then by Theorem \ref{components} the element $P_A (x)$ is given by $P_A (x)=\sup \{v\in A:0\leq v\leq x\}$. Since every element can be written as a difference of positive elements, we obtain the projection of an arbitrary element.

\begin{thm}\label{compot_as_projections}
If $V=A\oplus A^{\perp}$ is a mixed order direct sum and $P_A$ and $P_{A^{\perp}}=I-P_A$ are the associated projections on $A$ and $A^{\perp}$, respectively, then for any $x\in V$ we have
$$
x=P_A (x)+ P_{A^{\perp}}(x) = P_A (\rupp{x}) - P_A (\llow{x}) + P_{A^{\perp}}(\rupp{x}) - P_{A^{\perp}}(\llow{x}), 
$$
or alternatively,
$$
x=P_A (x)+ P_{A^{\perp}}(x) = P_A (\lupp{x}) - P_A (\rlow{x}) + P_{A^{\perp}}(\lupp{x}) - P_{A^{\perp}}(\rlow{x}). 
$$
\end{thm}

\begin{bproof}
Every $x\in V$ can be written as $x=\rupp{x}-\llow{x}$ (or alternatively, $x=\lupp{x}-\rlow{x}$, but this case is treated similarly).
The elements $\rupp{x}$ and $\llow{x}$ have the components $\rupp{x}=a+b$ and $\llow{x}=u+v$, where $a,u\in A$ and $b,v\in A^{\perp}$. 
On the other hand, $x=x_1+x_2$ with the components $x_1\in A$ and $x_2\in A^{\perp}$. By Theorem \ref{disjoint_sav} we have
$a+b=\rupp{x}=\rupp{x_1}+\rupp{x_2}$ and $u+v=\llow{x}=\llow{x_1}+\llow{x_2}$, where $\rupp{x_1},\llow{x_1}\in A$ and $\rupp{x_2},\llow{x_2}\in A^{\perp}$. 
Since $A\cap A^{\perp}=\{0\}$, this implies that $\rupp{x_1}=a$, $\rupp{x_2}=b$, $\llow{x_1}=u$ and $\llow{x_2}=v$. These components are the projection elements, and the proof is thus complete.
\end{bproof}

Next we consider some examples. First it should be noted that there are non-trivial mixed lattice spaces in which non-trivial ideals do not exist, and hence non-trivial mixed lattice decompositions do not always exist either.

\begin{example2}\label{nonontrivialideals}
Let $V=\R^2$ and define $\leq$ as the partial ordering with the usual positive cone $V_p=\{(x,y):x\geq 0, y\geq 0\}$. Let $\sleq$ to be the partial ordering induced by the positive cone $V_{sp}=\{\alpha(2,1)+\beta(1,2): \, \alpha,\beta\geq 0\}$. Then $V$ is a regular mixed lattice space. If $C_1=\{(x,y):y=\frac{1}{2}x\}$ and $C_2=\{(x,y):y=2x\}$ then $C_1$ and $C_2$ are both specific ideals, but there are no other ideals in $V$ than $\{0\}$ and $V$ itself. 

If we change $\sleq$ to be the partial ordering with the positive cone $V_{sp}=\{\alpha(1,0)+\beta(1,1): \, \alpha,\beta\geq 0\}$ then $V$ is again a regular mixed lattice space. Let $A=\{(x,y):y=x\}$ 
and $B=\{(x,y):y=0\}$. 
It is easy to see that $A$ is a regular specific ideal and $B$ is a regular ideal such that $V=A\oplus B$. The conditions of Theorem \ref{direct_sum} are thus satisfied and $A= {}^{\perp}B$ and $B= {A}^{\perp}$. Moreover,  $(A+B)_{sp}=A_{sp} + B_{sp}$ holds. However, this is not a mixed-order direct sum. For instance, the element $x=(1,2)$ cannot be written as $x=x_1+x_2$ where $x_1\in A_{p}$ and $x_2\in B_p$.
\end{example2}

The next one is related to the setting of Theorem \ref{direct_sum}.

\begin{example2}\label{bv1}
This example is adapted from \cite[pp. 34]{ars}. See also \cite[Example 2.15]{jj1}.
Let $V=BV([0,1])$ be the set of all functions of bounded variation on the interval $[0,1]$. Define the initial order as 
$f\leq g$  if  $f(x)\leq g(x)$ for all  $x\in [0,1]$, and the specific order as 
$f\sleq g$ if $f\leq g$ and  $g-f$ is non-decreasing on $[0,1]$. 
Then $V$ is a regular mixed lattice space.

Now fix some $c\in (0,1)$ and let $A$ be the subspace consisting of those functions that are constant on the closed interval $[c,1]$, and let $B$ be the subspace consisting of those functions that vanish on the closed interval $[0,c]$. Then $A$ is a regular specific ideal and $B$ is a regular ideal such that $A\cap B=\{0\}$. 
It is well known that every $g\in V$ can be written as a difference of two non-decreasing non-negative functions on $[0,1]$. Hence, to see that $V=A\oplus B$ it is sufficient to note that every non-decreasing non-negative function $f$ on $[0,1]$ can be written as $f=f_1+f_2$ where $f_1\in A$ and $f_2\in B$. Indeed, define $f_1$ by $f_1(x)=f(x)$ for all $x\in [0,c]$, and $f_1(x)=f(c)$ for all $x\in [c,1]$. Then define $f_2$ by $f_2(x)=0$ for all $x\in [0,c]$ and $f_2(x)=f(x)-f(c)$ for all $x\in [c,1]$. Then $f_1\in A$, $f_2\in B$ and $f=f_1+f_2$, and by Theorem \ref{direct_sum} $A=\ldc{B}$ is a specific band and $B=A^{\perp}$ is a weak band. 

\blue{
If we consider the space $W$ of all continuous functions of bounded variation on $[0,1]$, and we put $c=0$ with $A$ and $B$ defined as above, then $A$ is just the set of all constant functions on $[0,1]$ and $B=\{g\in W:g(0)=0\}$. Then $W=A\oplus B$ as above, and $B$ is a weak band by Theorem \ref{direct_sum}, but not a band. This can be seen by choosing $f(x)=1$ for all $x\in [0,1]$ and defining $f_n$ by $f_n(x)=nx$ for $x\in [0,1/n]$, and $f_n(x)=1$ for $x\in (1/n,1]$. Then $\{f_n\}\subseteq B$ for all $n\in \N$ and $\sup \{f_n\}=f$, but $f\notin B$. Note also, that now $f$ is not the specific supremum (and hence not the strong supremum) of $\{f_n\}$.}
\end{example2}

Special cases of mixed-order direct sums are provided by Dedekind complete Riesz spaces, where all bands are projection bands. 
We give some other examples below. %

\begin{example2}
Let $U$ be the same mixed lattice space as in Example \ref{rdisjoint_esim}, with the 
subspaces $C=\{(0,y,z):y,z\in \R \}$ and $D=\{(x,x,0):x\in \R\}$. Now $C$ is a regular ideal and $D$ is a regular specific ideal such that $U=C\oplus D$. Hence, the conditions of Theorem \ref{direct_sum} are satisfied and we have $C=D^{\perp}$ and $D= {}^{\perp}C$. Moreover, $(C+D)_{sp}=C_{sp} + D_{sp}$ but this is not a mixed-order direct sum. If $A=\{(0,0,z):z\in \R\}$ and $B=\{(x,y,0):x,y\in \R \}$  then $A$ and $B$ are both regular ideals such that $U=A\oplus B$. This is a mixed-order direct sum where $B=A^{\perp}$ and $A= {}^{\perp}B$.

Similarly, if $V$ is the same as in Example \ref{rdisjoint_esim}, then in $V$, $A$ is a regular band and $B$ is a band such that $B=A^{\perp}$, $A= {}^{\perp}B$ and $V=A\oplus B$ is a mixed order direct sum.
\end{example2}

\begin{example2}
Let $V$ be the set of all $n$-by-$n$ matrices and define $\leq$ as the usual element-wise ordering and define specific order by $M\sgeq 0$ if $M\geq 0$ and $M$ is a symmetric matrix. Then $V$ is a quasi-regular mixed lattice space (see \cite[Example 2.19]{jj1}), where the set $A$ consisting of all diagonal matrices is a band and the set $B$ of those matrices with zero diagonal elements is a band such that $A={}^{\perp}B$, \, $B=A^{\perp}$ and $V=A\oplus B$ is a mixed order direct sum.
\end{example2}

As we have seen in Example \ref{bv1}, the mixed lattice space of functions of bounded variation has direct sum decompositions of the type described in Theorem \ref{direct_sum}, but it does not possess non-trivial mixed-order direct sum decompositions. The following example is somewhat analogous to the situation in the Riesz space $C([0,1])$ of continuous real functions on the interval $[0,1]$, where non-trivial projection bands do not exist. Before discussing the example we need the following lemma.

\begin{lemma2}
The ideal $I(u)$ generated by a single element $u\in V$ is given by 
$I(u)=\{x\in V: s(x)\leq n s(u) \textrm{ for some } n\in \N\}$.
\end{lemma2}

\begin{bproof}
We first show that $I(u)$ is a subspace. If $x,y\in I(u)$ then $s(x)\leq n s(u)$ and $s(y)\leq m s(u)$ for some $n,m\in \N$. Then for all $a,b\in \R$,
$$
s(ax+by)\leq |a|s(x)+|b|s(y)\leq (|a|n+|b|m)s(u)\leq p s(u),
$$
where $p\in \N$ is a number such that $|a|n+|b|m\leq p$. Thus, $ax+ny\in I(u)$. Next, let $x\in I(u)$ and $s(y)\leq s(x)$. Then $s(y)\leq s(x)\leq n s(u)$ for some $n\in \N$ and so $y\in I(u)$. It follows that $I(u)$ is an ideal, by the condition of Theorem \ref{ideal_charact2}. To show that $I(u)$ is the smallest ideal that contains $u$, let $J$ be another ideal such that $u\in J$. Then also $ns(u)\in J$ for all $n\in \N$, and if $x\in I(u)$ the inequality $s(x)\leq ms(u)=s(mu)$ holds for some $m\in \N$. It follows again by Theorem \ref{ideal_charact2} that $x\in J$. Hence, $I(u)\subseteq J$.
\end{bproof}

\begin{example2}
Consider the regular mixed lattice space $V=BV([0,1])$, as in Example \ref{bv1}. If $f,g\in V$ then the mixed lower and upper envelopes are given by (\cite[Theorem 21.1.]{ars})
$$
(f\lenv g)(u)=\inf \, \{f(u)- (f(x)-g(x))^+ \, : \, x\in[0,u] \}
$$
and
$$
(f\uenv g)(u)=\sup \, \{f(u) + (g(x)-f(x))^+ \, : \, x\in[0,u] \},
$$
where $r^+ = \max \{0,r\}$ is the positive part of the real number $r$.

Now $V=V \oplus \{0\}$ is the only mixed order direct sum decomposition of $V$. This can be seen by considering the constant function $h(x)=1$ for all $x\in [0,1]$. First we note that the ideal generated by $h$ is $V$. Indeed, since every $f\in V$ is bounded, then for any $f\in V$ there exists some $n\in \N$ such that $s(f)\leq n h =ns(h)$. Hence $I(h)=V$, by the preceding lemma. 

Now if $V=A\oplus B$ is a mixed order direct sum then $h$ has the components $f\in A$ and $g\in B$ such that $f(x)+g(x)=1$ for all $x\in [0,1]$ and $f\lenv g =0$. This implies that $f$ and $g$ are positive, so $0\leq f(x)\leq 1$ and $0\leq g(x)\leq 1$ for all $x\in [0,1]$.  
Then the above expression for $f\lenv g$ gives 
$$
(f\lenv g)(u)=f(u)-\sup\{(f(x)-g(x))^+ : x\in[0,u]\}.
$$
If $f\lenv g=0$ then we have $f(u)=\sup\{(f(x)-g(x))^+ : x\in[0,u]\}$. In particular, $f(0)=(f(0)-g(0))^+=(2f(0)-1)^+$. Now, if $0\leq f(0)\leq \frac{1}{2}$ then $f(0)=0$. If $\frac{1}{2}< f(0)\leq 1$ then $f(0)=2f(0)-1$, or $f(0)=1$. Hence, we must have either $f(0)=0$ and $g(0)=1$, or $f(0)=1$ and $g(0)=0$. Since $h\sgeq 0$ and $V=A\oplus B$ is a mixed order direct sum then also the components of $h$ satisfy 
$f\sgeq 0$ and $g\sgeq 0$. In other words, $f$ and $g$ are non-decreasing. But this and $f(x)+g(x)=1$ imply that either $f(x)=0$ and $g(x)=1$, or $f(x)=1$ and $g(x)=0$ for all $x\in [0,1]$. This shows that either $A$ or $B$ contains the constant function $h(x)=1$. Since the ideal generated by $h$ equals $V$, we must therefore have either 
$A=V$ and $B=\{0\}$, or $B=V$ and $A=\{0\}$.
\end{example2}


A few concluding remarks are in order to further explain and justify our choice of certain definitions. Since we have defined $A^{\perp}$ to be the ideal generated by the cone $\aspp$, one might be inclined to ask why did we define $^{\perp}A$ as the specific ideal generated by the cone $\apld$, and not the ideal generated by $\apld$. The main reason for this stems from the fact that if $B$ is the ideal generated by $\apld$ then, in general, $B_{sp}$ is a larger set than $\apld$. As a consequence, properties such as $A\cap {}^{\perp}A = \{0\}$ would no longer hold. 
Moreover (and perhaps most importantly), there is a rather well developed theory of direct sum decompositions in mixed lattice semigroups, as presented in \cite{ars}. Indeed, our Theorems \ref{direct_sum} and \ref{positive_decomp2} have their counterparts in the theory of mixed lattice semigroups (\cite[Theorems 7.1 and 7.2]{ars}). (Note however, that the authors in \cite{ars} use different terminology, as they use the potential-theoretic notions of \emph{pre-harmonic band} and \emph{potential band}. The corresponding objects in this paper are called specific bands and weak bands, respectively.) Our present definitions of $A^{\perp}$ and $^{\perp}A$ are in agreement with the existing theory of mixed lattice semigroups. In fact, if $A$ is an ideal %
such that $V= {}^{\perp}A\oplus A$, then the set $V_{sp}$ is a mixed lattice semigroup in its own right (with the orderings inherited from $V$), and $V_{sp}=(^{\perp}A)_{sp} \oplus A_{sp}$ is the corresponding mixed lattice semigroup decomposition of $V_{sp}$, by Theorem \ref{direct_sum}.

With all these considerations, our definitions of the disjoint complements indeed seem to be the most natural, and also compatible with the theory of mixed lattice semigroups, at least if we restrict ourselves to mixed lattice subspaces, as we have done in this paper. However, the corresponding definitions for more general sets would be more problematic, as pointed out in Remark \ref{remark_leftdisjcompl}.





\bibliographystyle{plain}

\end{document}